\documentclass[reqno]{amsart}
\usepackage{graphicx} 
\usepackage{amsfonts}
\usepackage{amsmath}
\usepackage{amssymb}
\usepackage{amsthm}
\usepackage{esint}
\usepackage{geometry}
\usepackage{musicography}
\newtheorem{theorem}{Theorem}[section]
\newtheorem{proposition}[theorem]{Proposition}
\newtheorem{definition}[theorem]{Definition}
\newtheorem{corollary}[theorem]{Corollary}
\newtheorem{lemma}[theorem]{Lemma}
\newtheorem{remark}[theorem]{Remark}
\usepackage{esint}
\usepackage{xcolor}
\usepackage{appendix}

\usepackage{comment}
\usepackage{hyperref}

\title{Regularity for the geodesic X-ray transform in nonsmooth geometry}

\author{Pieti Kirkkopelto}
\address{Department of Mathematics and Statistics, University of Jyv\"askyl\"a, Jyv\"askyl\"a, Finland}
\email{pieti.e.kirkkopelto@jyu.fi}

\author{Miika Manu}
\address{Department of Mathematics and Statistics, University of Jyv\"askyl\"a, Jyv\"askyl\"a, Finland}
\email{miika.a.manu@jyu.fi}

\author{Mikko Salo}
\address{Department of Mathematics and Statistics, University of Jyv\"askyl\"a, Jyv\"askyl\"a, Finland}
\email{mikko.j.salo@jyu.fi}

\begin{document}

\begin{abstract}
    In this article we give regularity results for the geodesic X-ray transform on nonsmooth simple manifolds. As an application, we improve previous injectivity results for the geodesic X-ray transform acting on $L^p$ functions. The results are based on  symbol smoothing arguments for the normal operator and mapping properties of pseudodifferential operators with low regularity symbols.
\end{abstract}

\maketitle

\section{Introduction} \label{sec_intro}

\subsection{Background and related work}

In this article we study the geodesic X-ray transform for Riemannian metrics with low regularity. A standard setting for such problems is given by a \emph{simple manifold}, which is a compact, connected and oriented Riemannian manifold $(M,g)$ with smooth boundary such that $M$ is simply connected, the boundary $\partial M$ is strictly convex (i.e.\ the second fundamental form of $\partial M$ is positive definite), and no geodesic segment has conjugate points. There are several equivalent definitions, and any simple manifold is diffeomorphic to a ball \cite{paternain2023geometric}. In this setting it is typically assumed that the Riemannian metric $g$ is $C^{\infty}$.

The \emph{geodesic X-ray transform} maps a function $f \in C(M)$ to its X-ray transform 
\[
If(x,v) = \int_0^{\tau(x,v)} f(\gamma_{x,v}(t)) \,\mathrm{d}t,
\]
where $(x,v) \in \partial SM = \{ (x,v) \in TM \,:\, x \in \partial M, \ |v|_g = 1 \}$, $\gamma_{x,v}(t)$ is the geodesic through $x$ in direction $v$ and $\tau(x,v)$ is the time when $\gamma_{x,v}$ exits $M$. Thus $If$ encodes the integrals of $f$ over maximal geodesics in $M$. If $g$ is the Euclidean metric, then $I$ is the standard X-ray transform that forms the basis of medical imaging methods such as CT and PET \cite{natterer2001mathematics}. In the case where $g$ corresponds to the sound speed of a medium such as the Earth, this transform arises in the travel time tomography problem (also called the boundary rigidity problem) of determining properties in the interior of Earth from travel times of seismic waves \cite{paternain2023geometric}. It also arises in connection with the Calder\'on problem in Electrical Impedance Tomography \cite{DKSU2009} and in transmission ultrasound tomography.

For the above reasons, there has been considerable interest in understanding invertibility properties of the geodesic X-ray transform. In the Euclidean case, invertibility was established in the classical work \cite{radon1917determination} and for radial sound speeds satisfying the Herglotz condition this was done in \cite{herglotz1907benndorfsche, romanov1967reconstructing}. In the setting of simple manifolds, injectivity of the geodesic X-ray transform is due to \cite{muhometov1977problem} based on energy estimates and Pestov identities. This has been extended to nonconvex manifolds \cite{guillarmou2021boundary} and geometries with hyperbolic trapping \cite{guillarmou2017lens}. In dimensions $n \geq 3$, there is a powerful method for injectivity results in geometries with strictly convex foliations \cite{uhlmann2016inverse}. There is a large literature on various aspects of geodesic X-ray transforms and we refer the reader to \cite{sharafutdinov1994integral, paternain2023geometric, ilmavirta2019integral, stefanov2019travel} for further details.

Many of the above results are concerned with the case where the metric $g$ has many derivatives, i.e.\ $g \in C^{\infty}$. For materials occurring in practice, such as the sound speed in the Earth, this may not be a practical assumption. Thus, it is of interest to understand minimal assumptions on $g$ to ensure injectivity of the geodesic X-ray transform. If $g$ is $C^{1,1}$, then the nonlinearity appearing in the geodesic equation is Lipschitz continuous and the existence and uniqueness theorem for ODEs ensures that there is a unique geodesic $\gamma_{x,v}$ through any point $x$ and direction $v$. On the other hand, if $g$ is only $C^{1,\alpha}$ with $\alpha < 1$, there may be branching of geodesics. The case $g \in C^{1,1}$ is therefore a reasonable low regularity setting to study.

The geodesic X-ray transform with low regularity metrics was studied in \cite{zbMATH07725202}, where it was proved that on simple manifolds with $C^{1,1}$ metric the geodesic X-ray transform is injective when acting on Lipschitz continuous functions. The proof was based on a Pestov identity in low regularity and the article included a careful definition of a $C^{1,1}$ simple manifold. A similar result for tensor fields on negatively curved manifolds was given in \cite{ilmavirta2024tensor}. However, for smooth Riemannian metrics the injectivity result is known for $L^2$ functions and not just for Lipschitz functions. The setting of $L^2$ functions includes the case of piecewise continuous (or constant) media, which is relevant for practical imaging purposes. For smooth metrics, injectivity on $L^2$ functions is based on the following regularity result: if $f \in L^2(M)$ satisfies $If = 0$, one proves that then $f$ has to be smooth. This uses the fact that $I^* I f = 0$, where the normal operator $I^* I$ is an elliptic pseudodifferential operator. Then one appeals to the injectivity result for smooth functions.

The geodesic X-ray transform acting on $L^2$ functions when $g$ is a simple metric with low regularity has been considered in \cite{ilmavirta2023microlocal}. There it was proved that if $g \in C^{n+8}$, where $n = \dim(M)$, then $I$ is injective on $L^2(M)$. This was based on a regularity result for functions $f$ in the kernel of $I$, showing that $I^* I$ is an elliptic pseudodifferential operator in a suitable nonsmooth calculus and therefore can be inverted up to a smoothing error in this calculus. However, in order to use this nonsmooth calculus in \cite{ilmavirta2023microlocal}, one had to assume quite many derivatives on $g$. One could expect that the assumption $g \in C^{n+8}$ mentioned above is far from optimal for this problem.

\subsection{Main results}

In this article, we will improve upon the regularity results of \cite{ilmavirta2023microlocal}. First we give our definition of a simple manifold.

\begin{definition} \label{def_ctwo_simple}
Let $g$ be a $C^2$ Riemannian metric in $\mathbb{R}^n$, and let $M \subset \mathbb{R}^n$ be the closure of a bounded domain with smooth boundary. We say that $(M,g)$ is simple if $\partial M$ is strictly convex (its second fundamental form is positive definite), for any $x \in M$ the map $\exp_x$ is a diffeomorphism onto $M$, and the same properties hold when $M$ is replaced by the closure of some bounded smooth domain containing $M$.
\end{definition}

In the case of $C^{\infty}$ metrics, there are many equivalent definitions for a simple manifold \cite[Section 3.8]{paternain2023geometric} and the above definition is one of them \cite[Proposition 3.8.7]{paternain2023geometric}. For low regularity metrics, related definitions may be found in \cite{zbMATH07725202, ilmavirta2023microlocal}. Our main interest is on regularity results for the X-ray transform rather than on equivalent definitions of low regularity simple manifolds, and thus we have chosen the strong definition above. Further facts related to low regularity metrics are given in Appendix \ref{sect:propofgeodesics}.

Our first main theorem is a regularity result showing that if an $L^p$ function satisfies $If = 0$, then $f$ necessarily has more regularity. We define the Sobolev spaces $W^{s,p} = H^{s,p}$ as in \cite{Taylor3}.

\begin{theorem} \label{thm_main1}
Let $(M,g)$ be a simple manifold with $C^{4+\ell,\alpha}$ Riemannian metric where $\ell \geq 1$ and $0 < \alpha < 1$. If $f \in L^p(M)$ where $1 < p < \infty$ satisfies $If = 0$, then $f \in W^{\ell,p}(M)$.
\end{theorem}

This improves the corresponding result in \cite{ilmavirta2023microlocal} that considered functions in $L^2(M)$ when $g \in C^k$ and $k \geq \ell + 6 + n/2$. The assumption that $g$ has more than four derivatives is used in proving sufficient smoothing estimates for the error term $R$ in \eqref{n_decomposition_first} and we do not know if this can be improved. Combining Theorem \ref{thm_main1} with the injectivity result on Lipschitz continuous functions from \cite{zbMATH07725202}, we obtain the following injectivity result.

\begin{theorem} \label{thm_main2}
Let $(M,g)$ be a simple manifold with $C^{5,\alpha}$ Riemannian metric where $0 < \alpha < 1$, and let $1 < p < \infty$. If $f \in L^p(M)$ satisfies $If = 0$, then $f = 0$.
\end{theorem}

Again, this extends the earlier result in \cite{ilmavirta2023microlocal} which assumed $g \in C^{n+8}$. We remark that if one would take a slightly weaker definition of a $C^2$ simple manifold ($\partial M$ is strictly convex and $\exp_x$ is a diffeomorphism onto $M$ for any $x \in M$), then Theorems \ref{thm_main1}--\ref{thm_main2} would hold for $f \in L^p$ with $\mathrm{supp}(f) \subset M^{\mathrm{int}}$. Moreover, given a $C^2$ analogue of \cite[Proposition 3.8.7]{paternain2023geometric}, the theorems would hold as stated even with this weaker definition.

Our proof of Theorem \ref{thm_main1} is still based on the normal operator $N = I^* I$: we show that $N$ is invertible modulo a smoothing term, and thus $If = 0$ implies that $f$ has additional regularity. To do this, we will write 
\[
N = P + R,
\]
where $P$ is an elliptic pseudodifferential operator whose symbol has finite regularity in $x$ and $R$ is a smoothing error term. The main point is that we will use a nonsmooth calculus and symbol smoothing arguments as in \cite[Section 13.9]{Taylor3} to reduce matters to an elliptic operator in the smooth calculus. We remark that the results in \cite{Taylor3} are formulated for the Zygmund spaces $C^r_*$, but since we only consider the case $r=k+\alpha$ with $k\in\mathbb{N}_{\ge 0}$ and $\alpha\in (0,1)$ and in this case $C^r_*\cong C^{k,\alpha}$, we will state our results for the Hölder spaces $C^{k,\alpha}$.

We explain the argument in somewhat more detail. Full details can be found in Section \ref{sect:proofsofmainresults}. By the definition of a simple manifold, we may assume that we are working in a global coordinate chart in $\mathbb{R}^n$. When $g \in C^{k,\alpha}$, where $k \ge 2$ and $(M,g)$ is simple, the normal operator has the explicit form 
\[
Nf(x) = \int_M \frac{a(x,y)}{d(x,y)^{n-1}} f(y) |g(y)|^{1/2} \,\mathrm{d}y,
\]
where $a(x,y) \in C^{k-2,\alpha}$ is related to the derivative of the exponential map with $a(x,x) = 1$, and $|g(x)| = \det(g(x))$. Moreover, using the simplicity assumption the Riemannian distance function $d(x,y) = d_g(x,y)$ takes the form 
\[
d(x,y) = \sum_{p,q=1}^n G_{pq}(x,y) (x-y)^p (x-y)^q
\]
where $G_{pq} \in C^{k-2,\alpha}$ and $G_{pq}(x,x) = g_{pq}(x)$ (see Lemma \ref{lemma_distance}). We extend $f$ by zero to a function on $\mathbb{R}^n$ and use suitable cutoff functions supported in the simple extension of $M$ to write the integral over $\mathbb{R}^n$.  Taking suitable Taylor expansions at $y=x$ lets us write 
\begin{equation} \label{n_decomposition_first}
N f = P_1 f + R f,
\end{equation}
where $P_1$ is the Riesz potential type operator 
\[
P_1 f(x) = c_n |g(x)|^{1/2} \int_{\mathbb{R}^n} \frac{1}{|x-y|_{g(x)}^{n+1}} f(y) \,\mathrm{d}y
\]
and $R$ is a smoothing operator. More precisely, the kernel $R(x,y)$ of $R$ is of class $C^{k-3,\alpha}$ and satisfies estimates like 
\[
|R(x,y)| \lesssim |x-y|^{2-n}, \qquad |\nabla R(x,y)| \lesssim |x-y|^{1-n}, \qquad |\nabla^2 R(x,y)| \lesssim |x-y|^{-n}
\]
if $k \ge 5$. We will prove by direct singular integral arguments that $R$ maps $L^p$ to $W^{2,p}$ for $1 < p < \infty$.

\begin{remark}
     We use the notation $\lesssim$ to denote inequalities holding with uniform constants and $\lesssim_s$ for constants depending on a parameter $s$.
\end{remark} 

Now we observe that $P_1$ is in fact a pseudodifferential operator, 
\[
P_1 f(x) = (2\pi)^{-n} \int_{\mathbb{R}^n} e^{ix \cdot \xi} p_1(x,\xi) \hat{f}(\xi) \,\mathrm{d}\xi
\]
where the (nonsmooth) symbol is given by 
\[
p_1(x,\xi) = c_n |\xi|_{g(x)}^{-1}.
\]
To deal with the singularity at $\xi=0$, we introduce a cutoff function $\psi \in C^{\infty}_0(\mathbb{R}^n)$ with $\psi = 1$ for $|\xi| \leq 1/2$ and $\psi = 0$ for $|\xi| \geq 1$, and define 
\[
p(x,\xi) = (1-\psi(\xi)) p_1(x,\xi).
\]
We define the nonsmooth symbol classes $C^{k,\alpha}S^m_{1,\delta}$ as follows. 

\begin{definition} \label{def_nonsmooth_symbol}
        Let $k\in \mathbb{N}_{\geq0}$, \(\alpha\in(0,1)\), $m \in \mathbb{R}$ and $0 < \delta \leq 1$. We say that \(p(x,\xi)\in C^{k,\alpha}S^m_{1,\delta}\) if for any multi-index $\beta$ there is $C_{\beta} > 0$ such that, with $\left< \xi \right> = (1+|\xi|^2)^{1/2}$,
        \begin{itemize}
            \item \(|D^{\beta}_\xi p(x,\xi)|\leq C_{\beta}\left<\xi\right>^{m-|\beta|}\)
            \item \(||D_\xi^\beta p(\cdot,\xi)||_{C^{k,\alpha}}\leq C_\beta\left<\xi\right>^{m-|\beta|+(k+\alpha)\delta}\)
        \end{itemize}
\end{definition}

For more details on these symbol classes see \cite[Section 13.9]{Taylor3}. Since $g \in C^{k,\alpha}$, we have $p \in C^{k,\alpha} S^{-1}_{1,0}$. It follows that 
\begin{equation} \label{n_decomposition_second}
N = P + R' + R
\end{equation}
for another smoothing operator $R'$ with kernel $R'(x,y) = (2\pi)^{-n} \int_{\mathbb{R}^n}e^{i(x-y) \cdot \xi} \psi(\xi) |\xi|_{g(x)}^{-1} \,\mathrm{d}\xi$. The operator $R'$ maps $L^p$ to $W^{\ell,p}$ for $\ell \leq k$ (recall that there are cutoffs in $x$ and $y$ that we have not written down). The operator $P$ has coefficients of class $C^{k,\alpha}$ and the operator $R$ has coefficients of class $C^{k-3,\alpha}$. Hence $R$ will be the most restrictive term.

Next we apply the symbol smoothing argument in \cite[equation (9.27)]{Taylor3}. Given any $\delta \in (0,1]$, we write 
\[
p = p^{\sharp} + p^{\flat}
\]
where $p^{\sharp} \in S^{-1}_{1,\delta}$ is a symbol that is smooth in $x$, and $p^{\flat} \in C^{k,\alpha} S^{-1-r \delta}_{1,\delta}$ is a lower order nonsmooth symbol that will result in another smoothing term.

The smooth symbol $p^{\sharp}$ turns out to be elliptic and hence there is $q \in S^{1}_{1,\delta}$ such that 
\[
QP^{\sharp} = I + S
\]
where $S$ has symbol $s \in S^{-\infty}$ by the smooth pseudodifferential calculus. It follows that 
\[
QN = I + Q P^{\flat} + Q R' + Q R + S.
\]
The second term is smoothing of any order $\leq k - 1$ by \cite[Section 13.9]{Taylor3} and the third term is smoothing of order $\ell -1$ when $\ell \leq k$. Recall that $\delta \in (0,1]$ was arbitrary. Thus $Nf = 0$ implies that 
\begin{equation}\label{minusf_eq1}
-f = (QP^{\flat} f + Q R' f + Sf) + Q R f.
\end{equation}
The term in parentheses is in $W^{k+\alpha-2-\varepsilon}$. Moreover, since $R$ maps $L^p$ to $W^{2,p}$ and $Q$ is of order one, we see that $f \in W^{1,p}$ if $k \geq 5$.

This argument can be iterated to show that if $g \in C^{k,\alpha}$ with $k \geq 4 + \ell$, and if $If = 0$ where $f \in L^p$, then $f \in W^{\ell,p}$.

\subsection{Organization of article}

Section \ref{sect:proofsofmainresults} is devoted to the proofs of the main results and Section \ref{section:mappingpropofremainderterms} includes the necessary technical results used to establish the mapping properties of the smoothing remainder terms. Appendix \ref{sect:propofgeodesics} contains details on geodesics of nonsmooth Riemannian manifolds.

\subsection*{Acknowledgements}

All authors were partly supported by the Research Council of Finland (Centre of Excellence in Inverse Modelling and Imaging and FAME Flagship, grants 353091, 353092 and 359208). P.K.\ and M.M.\ were also supported by the Finnish Ministry of Education and Culture’s Pilot for Doctoral Programmes (Pilot project Mathematics of Sensing, Imaging and Modeling).

\section{Proofs of the main results}\label{sect:proofsofmainresults}

We will now give the proofs of the main theorems, based on a number of auxiliary results that will be stated and proved in Section \ref{section:mappingpropofremainderterms}.

\begin{proof}[Proof of Theorem \ref{thm_main1}]
We repeat the argument in Section \ref{sec_intro} with more details. Let $g\in C^{k,\alpha}$ for some $k\ge5$. We begin with the decomposition given in \eqref{n_decomposition_second}, i.e. 
\[Nf(x)=Pf(x)+Rf(x)+R'f(x).\]
The details are given by Lemma \ref{decomposition}.
We can then apply the symbol smoothing argument in \cite[Section 13.9]{Taylor3} to write for any $\delta\in (0,1]$ that
\begin{equation*}
    p=p^\sharp+p^\flat,
\end{equation*}
where $p^\sharp\in S^{-1}_{1,\delta}$ and $p^\flat\in C^{k,\alpha}S^{-1-r\delta}_{1,\delta}$. The symbol $p^{\sharp}$ is elliptic, since for $|\xi| \geq 1$ we have 
\begin{equation*}
    \left\vert p^{\sharp}\right\vert = \left\vert p - p^{\flat}\right\vert \geq c_n \left\vert\xi\right\vert_{g(x)}^{-1} - C \left\vert \xi \right\vert^{-1-r\delta} = \vert\xi\vert^{-1} \left( c_n \frac{\vert\xi\vert}{\vert\xi\vert_{g(x)}} - C \vert\xi\vert^{-r\delta} \right) 
\end{equation*}
and the term in brackets is bounded from below by some $c>0$ uniformly over $x$ when $\xi$ is sufficiently large.

Since $p^{\sharp}$ is elliptic, there is $q \in S^{1}_{1,\delta}$ such that, denoting $Q:=\mathrm{Op}(q)$, we have 
\[
QP^{\sharp} = I + S
\]
where $S$ has symbol $s \in S^{-\infty}$ by the smooth pseudodifferential calculus. It follows that 
\[
QN = I + Q P^{\flat} + Q R' + Q R + S.
\]
The second term is smoothing of any order strictly less than $r - 1$ by \cite[Section 13.9]{Taylor3} and by Lemma \ref{arpilkku} the third term is smoothing of order $\ell -1$ when $\ell \leq k$. Recall that $\delta \in (0,1]$ was arbitrary. Thus $Nf = 0$ implies that 
\begin{equation} \label{minusf_eq}
-f = (QP^{\flat} f + Q R' f + Sf) + Q R f,
\end{equation}
where the term in parentheses is in $W^{k+\alpha-2-\varepsilon,p}$ for any $\varepsilon>0$ by making a suitable choice of $\delta$.

Suppose then that $k\ge 5$. By Corollary \ref{rmappingproperties} the operator $R$  maps $L^p$ to $W^{2,p}$ and since $Q$ is a pseudodifferential operator of order one, we see that $f \in W^{1,p}$ if $k\ge 5$.

Let us now assume $k\ge  6$. We prove that $If = 0$ for $f \in L^p$ implies $f \in W^{2,p}$. In \eqref{minusf_eq}, the term in parentheses is already in $W^{k+\alpha-2-\varepsilon,p}$ and thus in $W^{4,p}$ for $\varepsilon>0$ small enough. We also already know that $f \in W^{1,p}$. It is thus enough to show that  if $f \in W^{1,p}$ and $k \ge 6$, then $R f \in W^{3,p}$. We then have that
\begin{equation*}
    \partial_jRf=R(\partial_jf)+[\partial_j,R]f,
\end{equation*}
where by Lemma \ref{commu} the operators \(R\) and $[\partial_j,R]$ map \(L^p\) to \(W^{2,p}\). Hence \(f\in W^{3,p}\).

By iterating this argument we can show that if $g \in C^{k,\alpha}$ with $k \ge 4 + \ell$, and if $If = 0$ where $f \in L^p$, then $f \in W^{\ell,p}$. We prove this by induction.

The case \(\ell=1\) was done above. Now assume that claim holds for every \(\ell'<\ell\). Especially we can choose \(\ell'=\ell-1\). Now by the induction assumption we have that \(f\in W^{\ell-1,p}\). Next we use Corollary \ref{rmappingproperties} to deduce that \(Rf\in W^{\ell+1,p}\), so therefore by \eqref{minusf_eq} we have that \(f\in W^{\ell,p}\), which proves the claim.
\end{proof}

As advertised, Theorem \ref{thm_main1} lets us obtain the injectivity result in Theorem \ref{thm_main2}.

\begin{proof}[Proof of Theorem \ref{thm_main2}]
Suppose $g\in C^{5,\alpha}(M)$. Let $p\in (1,\infty)$ and suppose \(f\in L^p(M)\) such that \(If=0\). Then by Theorem \ref{thm_main1} we have $f \in W^{1,p}$. If $p > n$, Sobolev embedding gives $f \in C^{\beta}$ for some $\beta > 0$. On the other hand if $p < n$ (for $p=n$ we can just decrease $p$ slightly), Theorem \ref{thm_main1} and Sobolev embedding give $f \in L^q$ for $q = \frac{np}{n-p} > p$. After iterating this argument finitely many times, we see that $f \in L^r$ for some $r > n$, and then Theorem \ref{thm_main1} and Sobolev embedding give $f \in C^{\beta}$ for some $\beta > 0$.

By Lemma \ref{höldermapping} with H\"older exponent $\gamma = \min\{ \alpha, \beta\} > 0$ we have that $f\in C^{1,\gamma}$ and hence that the function $f$ is Lipschitz. By the injectivity result for Lipschitz functions in \cite{zbMATH07725202} this lets us conclude that $f=0$.
\end{proof}

\section{Technical results}\label{section:mappingpropofremainderterms}

Let $(M,g)$ be a $C^2$ simple manifold with $M \subset \mathbb{R}^n$, and let $(\tilde{M},g)$ with $M \subset \tilde{M}^{\mathrm{int}}$ be a slightly larger simple manifold as in Definition \ref{def_ctwo_simple}. We begin by reducing the normal operator $N$ to one acting on $C^\infty_0(\mathbb{R}^n)$. Let $\psi\in C^\infty_0(\mathbb{R}^n)$ be a smooth cutoff function such that $\psi=1$ in $M$ and $\mathrm{supp}(\psi)\Subset \tilde{M}$. We identify functions $f\in C^\infty_0(M)$ with their zero extensions to $C_0^\infty(\mathbb{R}^n)$ and write 
\begin{equation*}
    (Nf)(x)=\int \frac{\psi(x)a(x,y)\vert g\vert^{1/2}(y)\psi(y)}{d(x,y)^{n-1}}f(y)\ \mathrm{d}y.
\end{equation*}
For details on the representation for the normal operator we refer to \cite[Lemma 8.1.10]{paternain2023geometric} in the $C^{\infty}$ case. The same argument works for $C^2$ metrics \cite[Lemma 8]{ilmavirta2023microlocal}.

\begin{remark}
    In this section we use the notation $\partial_i^\alpha$ to denote derivatives with respect to the $i$:th argument, $\partial_{i,j}$ the $j$:th partial derivative with respect to the $i$:th argument and for notational convenience we omit the integration domain from integrals over the whole space $\mathbb{R}^n$. 
\end{remark}

We then have the following decomposition.

\begin{lemma}\label{decomposition}
    Let $f\in C^\infty_0(M)$ and identify $f$ with its zero extension to $\mathbb{R}^n$. Then
    \begin{equation*}
        Nf=(P+R'+R)f,
    \end{equation*}
    where the operator $P$ is defined by  
    \begin{equation*}
        (Pf)(x)=c_n\int e^{i\langle x,\xi\rangle}\psi^2(x)\vert\xi\vert_{g(x)}^{-1}(1-\tilde{\psi}(\xi))\widehat{f}(\xi)\ \mathrm{d}\xi,
    \end{equation*}
    with symbol $p(x,\xi) = c_n \psi(x)^2 (1-\tilde{\psi}(\xi)) |\xi|_{g(x)}^{-1}$, 
    the operator $R'$ by
    \begin{equation*}
        (R'f)(x)=c_n\int e^{i\langle x,\xi\rangle}\psi^2(x) \vert\xi\vert_{g(x)}^{-1}\tilde{\psi}(\xi)\widehat{f}(\xi)\ \mathrm{d}\xi
    \end{equation*}
    and the operator $R$ by
    \begin{equation*}
        (Rf)(x)=\int\left(\frac{\psi(x)a(x,y)\vert g\vert^{1/2}(x)\psi(y)}{d(x,y)^{n-1}}-\frac{\psi^2(x)\vert g\vert^{1/2}(x)}{\vert x-y\vert_{g(x)}^{n-1}}\right)\psi(y)f(y)\ \mathrm{d}y.
    \end{equation*}
\end{lemma}
\begin{proof}
    We write
    \begin{equation*}
    \begin{split}
        &(Nf)(x) \\
        &=\int \frac{\psi^2(x)\vert g\vert^{1/2}(x)}{\vert x-y\vert_{g(x)}^{n-1}}f(y)\ \mathrm{d}y+\int\left(\frac{\psi(x)a(x,y)\vert g\vert^{1/2}(y)\psi(y)}{d(x,y)^{n-1}}-\frac{\psi^2(x)\vert g\vert^{1/2}(x)}{\vert x-y\vert_{g(x)}^{n-1}}\right)f(y)\ \mathrm{d}y \\
        &=\int \frac{\psi^2(x)\vert g\vert^{1/2}(x)}{\vert x-y\vert_{g(x)}^{n-1}}f(y)\ \mathrm{d}y+\int\left(\frac{\psi(x)a(x,y)\vert g\vert^{1/2}(y)\psi(y)}{d(x,y)^{n-1}}-\frac{\psi^2(x)\vert g\vert^{1/2}(x)}{\vert x-y\vert_{g(x)}^{n-1}}\right)\psi(y)f(y)\ \mathrm{d}y \\
        &=:(P_1f)(x)+(Rf)(x).
    \end{split}
\end{equation*}
Let then $\tilde{\psi}$ be a smooth cutoff function with $\tilde{\psi}=1$ in $\Bar{B}(0;1/2)$ and $\mathrm{supp}(\tilde{\psi})\subset B(0;1)$. By using the Fourier transform identities $\int f(y) \overline{h(y)} \ \mathrm{d}y = (2 \pi)^{-n} \int \hat{f}(\xi) \overline{\hat{h}(\xi)} \,d\xi$ with $h(y) = |x-\bullet|_{g(x)}^{1-n}$ and 
\[
\hat{h}(\xi) = e^{-ix \cdot \xi} |g(x)|^{-1/2} \frac{c_n}{|\xi|_{g(x)}}
\]
where $c_n$ is a dimension-dependent constant, for $f\in C^\infty_0(M)$ we can write
\begin{equation*}
    \begin{split}
        (P_1f)(x)
        &=\frac{1}{(2\pi)^n}\int e^{i\langle x,\xi\rangle}c_n\psi^2(x) \vert \xi\vert_{g(x)}^{-1}\widehat{f}(\xi)\ \mathrm{d}\xi \\
        &=\frac{1}{(2\pi)^n}\int e^{i\langle x,\xi\rangle}c_n\psi^2(x)  \vert \xi\vert_{g(x)}^{-1}(1-\tilde{\psi}(\xi))\widehat{f}(\xi)\ \mathrm{d}\xi \\
        &\quad +\frac{1}{(2\pi)^n}\int e^{i\langle x,\xi\rangle}c_n\psi^2(x) \vert \xi\vert_{g(x)}^{-1}\tilde{\psi}(\xi)\widehat{f}(\xi)\ \mathrm{d}\xi \\
        &=:(Pf)(x)+(R'f)(x). \qedhere
    \end{split}
\end{equation*}
\end{proof}

The symbol $p$ in Lemma \ref{decomposition} belongs to the nonsmooth symbol class in Definition \ref{def_nonsmooth_symbol}:

\begin{lemma}\label{symbolclass}
    Suppose $g\in C^{k,\alpha}(M)$. Then $p\in C^{k,\alpha}S^{-1}_{1,0}$.
\end{lemma}
\begin{proof}
Since the function
    \begin{equation*}
        x\mapsto c_n\psi^2(x)
    \end{equation*}
    is compactly supported and smooth, it suffices to verify the decay estimates for the derivatives of
    \begin{equation*}
        \tilde{p}(x,\xi) = \vert \xi\vert_{g(x)}^{-1}(1-\tilde{\psi}(\xi))
    \end{equation*}
    when $x$ is in $\mathrm{supp}(\psi)$. We begin by noting that $\tilde{p}$ is $(-1)$-homogeneous when $\vert \xi\vert\ge 1$ and smooth and bounded in $\xi$ when $\vert \xi\vert<1$. Hence
\begin{equation*}
    \vert \partial_{2}^\gamma \tilde{p}(\bullet,\xi)\vert \leq C_\gamma\langle\xi\rangle^{-1-\vert\gamma\vert}.
\end{equation*}

Suppose $\vert\beta\vert=k$. Then for all $\xi\in B(0;1)^c$ and all $\gamma$ we have by homogeneity that
\begin{equation*}
    \begin{split}
        \frac{\vert\partial_1^\beta\partial_2^\gamma \tilde{p}(x,\xi)-\partial_1^\beta \partial_2^\gamma \tilde{p}(y,\xi)\vert}{\vert x-y\vert^\alpha}&=\frac{\vert\partial_1^\beta\partial_2^\gamma \tilde{p}(x,\xi/\vert\xi\vert)-\partial_1^\beta \partial_2^\gamma \tilde{p}(y,\xi/\vert\xi\vert)\vert}{\vert x-y\vert^{\alpha}}\vert\xi\vert^{-1-\vert\gamma\vert}.
    \end{split}
\end{equation*}
When $\vert\xi\vert\ge c>0$, we have
\begin{equation*}
    \vert\xi\vert^{-1-\vert\gamma\vert}\lesssim (1+\vert\xi\vert)^{-1-\vert\gamma\vert}\lesssim \langle\xi\rangle^{-1-\vert\gamma\vert}.
\end{equation*}
Since Hölder regularity is preserved by compositions with smooth functions from the left, we have by the fact that $g\in C^{k,\alpha}$ the estimate
\begin{equation*}
    \sup_{x\neq y}\frac{\vert\partial_1^\beta\partial_2^\gamma \tilde{p}(x,\xi/\vert\xi\vert)-\partial_1^\beta \partial_2^\gamma \tilde{p}(y,\xi/\vert\xi\vert)\vert}{\vert x-y\vert^\alpha}\vert\xi\vert^{-1-\vert\gamma\vert}\lesssim_\gamma\langle\xi\rangle^{-1-\vert\gamma\vert}.
\end{equation*}
Hence
\begin{equation*}
    \Vert \partial_2^\gamma \tilde{p}(\bullet,\xi)\Vert_{C^{k,\alpha}}\lesssim_\gamma \langle\xi\rangle^{-1-\vert\gamma\vert},
\end{equation*}
giving us the result.
\end{proof}

\subsection{Mapping properties for the remainder terms: Sobolev estimates}

The remainder operator $R'$ in Lemma \ref{decomposition} has the following mapping properties.

\begin{lemma}\label{arpilkku}
    Suppose $g\in C^k$ for some $k\in \mathbb{N}_{\ge 1}$. Then $R'f\in C^k_0$ for all $f\in L^p(M)$ with $\Vert R'f\Vert_{W^{k,p}}\lesssim \Vert f\Vert_{L^p}$.
\end{lemma}
\begin{proof}
    Suppose $g\in C^k$ with $k\ge 1$, let $f\in L^p(M)$ and identify $f$ with its zero extension to $\mathbb{R}^n$. The Schwartz kernel of $R'$ is given by the absolutely convergent integral
    \begin{equation*}
        K_{R'}(x,y) = \frac{1}{(2\pi)^n}\int e^{i\langle x-y,\xi\rangle}c_n\psi^2(x) \vert \xi\vert_{g(x)}^{-1}\tilde{\psi}(\xi)\ \mathrm{d}\xi.
    \end{equation*}
    Since $f\in L^p(M)$, we can identify $R'$ with the operator with the compactly supported Schwartz kernel
    \begin{equation*}
        \tilde{K}_{R'}(x,y) = K_{R'}(x,y)\psi(y).
    \end{equation*}
    We claim that $\tilde{K}_{R'}$ is in $C^k_0(\mathbb{R}^n\times\mathbb{R}^n)$.

    To show this, we note that it suffices to show that the function
    \begin{equation*}
        (x,y)\mapsto \int e^{i\langle x-y,\xi\rangle}\vert \xi\vert_{g(x)}^{-1}\tilde{\psi}(\xi)\ \mathrm{d}\xi
    \end{equation*}
    is $C^k$ when $x$ and $y$ range over a compact set. Since $g$ is $C^k$, for all $\alpha$ and $\beta$ with $\vert\alpha\vert+\vert\beta\vert\leq k$ we have 
    \begin{equation*}
        \vert\partial_x^\alpha\partial_y^\beta( e^{i\langle x-y,\xi\rangle}\vert \xi\vert_{g(x)}^{-1}\tilde{\psi}(\xi))\vert\lesssim \vert \xi\vert^{-1}\vert \eta(\xi)\vert
    \end{equation*}
    for some compactly supported smooth function $\eta$. Since the function
    \begin{equation*}
        \xi\mapsto \vert\xi\vert^{-1}\vert\eta(\xi)\vert
    \end{equation*}
    is integrable, we can differentiate under the integral sign, giving us the claim.

    Since $\tilde{K}_{R'}\in C^k_0(\mathbb{R}^n\times\mathbb{R}^n)$, we have by Hölder's inequality that
    \begin{equation*}
        \Vert \partial^\alpha R'f\Vert_{L^p}\lesssim \Vert f\Vert_{L^p}
    \end{equation*}
    for all $\alpha$ with $\vert\alpha\vert\leq k$, giving us the claim.
\end{proof}

We have the following representation result for the Schwartz kernel of the operator $R$ in Lemma \ref{decomposition}.

\begin{lemma}\label{arrankerneli}
    Suppose $g\in C^k$ for some $k\ge 3$. Then
    \begin{equation*}
        \frac{\psi(x)\psi(y)a(x,y)\vert g(x)\vert^{1/2}}{d(x,y)^{n-1}}-\frac{\psi^2(x)\vert g(x)\vert^{1/2}}{\vert x-y\vert_{g(x)}^{n-1}}=K(x,y,x-y),
    \end{equation*}
    where the function $K:\mathbb{R}^n\times\mathbb{R}^n\times(\mathbb{R}^n\setminus\{0\})\to\mathbb{R}$ is
    \begin{itemize}
        \item of class $C^{k-3}$ in the first two arguments,
        \item compactly supported in the first argument and
        \item smooth outside the origin, odd and positively $-(n-2)$-homogeneous in the last argument.
    \end{itemize}
\end{lemma}
\begin{proof}
    By Lemma \ref{lemma_distance} we can write
    \begin{equation*}
        \frac{\psi(x)\psi(y)a(x,y)\vert g(x)\vert^{1/2}}{d(x,y)^{n-1}}=\frac{\psi(x)\psi(y)a(x,y)\vert g(x)\vert^{1/2}}{(G_{ij}(x,y)(x-y)^i(x-y)^j)^{(n-1)/2}}.
    \end{equation*}
    We use the Einstein summation convention, i.e.\ a repeated index in both lower and upper position will be summed from $1$ to $n$. Since $g\in C^{k}$, we have by Lemmas \ref{lemma_distance} and \ref{lemma_geodesic_ckminusone} that $a,G\in C^{k-2}$ and hence we can take the Taylor expansion in the second argument of the function
    \begin{equation*}
        \Tilde{K}:(x,y,z)\mapsto \frac{\psi(x)\psi(y)a(x,y)\vert g(x)\vert^{1/2}}{(G_{ij}(x,y)z^i z^j)^{(n-1)/2}}
    \end{equation*}
    at $y=x$
    to write
    \begin{equation*}
        \frac{\psi(x)\psi(y)a(x,y)\vert g(x)\vert^{1/2}}{d(x,y)^{n-1}}=\frac{\psi^2(x)\vert g(x)\vert^{1/2}}{\vert x-y\vert_{g(x)}^{n-1}}+\left\langle \int_0^1(\nabla_2\Tilde{K})(x,x+t(y-x),x-y)\ \mathrm{d}t, y-x \right\rangle.
    \end{equation*}
    We define the kernel $K$ by
    \begin{equation*}
        K(x,y,z)=-\left\langle \int_0^1(\nabla_2\Tilde{K})(x,x+t(y-x),z)\ \mathrm{d}t,z\right\rangle
    \end{equation*}
    and note that now
    \begin{equation*}
        \frac{\psi(x)\psi(y)a(x,y)\vert g(x)\vert^{1/2}}{d(x,y)^{n-1}}-\frac{\psi^2(x)\vert g(x)\vert^{1/2}}{\vert x-y\vert_{g(x)}^{n-1}}=K(x,y,x-y).
    \end{equation*}
    
    Since the function $\Tilde{K}$ is even and positively $-(n-1)$-homogeneous in the last argument, so is the function $\nabla_2\Tilde{K}$. Hence the function $K$ is odd and positively $-(n-2)$-homogeneous in the last argument. Since the function $\psi$ is compactly supported, the function $K$ is compactly supported in the first argument. Regularity follows from the regularity of the functions $a$ and $G$.
\end{proof}

\begin{remark}
    Since our function $f$ is supported in $M$, we may replace the kernel $K$ with the kernel
\begin{equation*}
    (x,y,z)\mapsto \psi'(y)K(x,y,z),
\end{equation*}
where $\psi'$ is a compactly supported smooth cutoff function with $\psi'=1$ in $M$. From now on we identify the Schwartz kernel of the operator $R$ with this modified kernel.
\end{remark}

We then claim the following.

\begin{proposition}\label{Highermapping}
   Suppose $K:\mathbb{R}^n\times\mathbb{R}^n\times(\mathbb{R}^n\setminus\{0\})\to \mathbb{R}$ is
    \begin{itemize}
        \item of class $C^{k+2}$ and compactly supported in the first two arguments and
        \item smooth outside the origin, odd and positively $-(n-2)$-homogeneous in the last argument.
    \end{itemize}
    Then the associated integral operator $T_K$, defined on $\mathcal{S}$ by the uniformly convergent integral
    \begin{equation*}
        (T_K\phi)(x)=\int K(x,y,x-y)\phi(y)\ \mathrm{d}y,
    \end{equation*}
    extends uniquely to a bounded operator $T_K:W^{l,p}\to W^{l+2,p}$ for $l\leq k$.
\end{proposition}

We split the proof into several steps. The first step involves a weakly singular integral operator.

\begin{lemma}\label{weaksingularpvalue}
    Let $p\in (1,\infty)$ and $s>0$. Suppose $K:\mathbb{R}^n\times\mathbb{R}^n\times(\mathbb{R}^n\setminus\{0\})\to \mathbb{R}$ is compactly supported and continuous with respect to the first two arguments and positively $-(n-s)$-homogeneous and continuous outside the origin with respect to the third argument. Then the associated integral operator $T_K$, defined on $\mathcal{S}$ via the absolutely convergent integral
    \begin{equation*}
        (T_K\phi)(x)=\int K(x,y,x-y)\phi(y)\ \mathrm{d}y,
    \end{equation*}
    extends uniquely to a bounded operator $T_K:L^p\to L^p$ and moreover
    \begin{equation*}
        (T_Kf)(x)=\lim_{\varepsilon\to 0}\int_{B^c(x;\varepsilon)}K(x,y,x-y)f(y)\ \mathrm{d}y,
    \end{equation*}
    where the limit exists in $L^p$ and pointwise almost everywhere.
\end{lemma}
\begin{proof}
The assumptions on $K$ imply that 
\[
|K(x,y,z)| \lesssim \chi(x) \chi(y) |z|^{-(n-s)}
\]
for some $\chi \in C^{\infty}_0(\mathbb{R}^n)$. Thus $K(x,y,x-y)$ satisfies the condition in Schur's test \cite[Theorem 6.18]{folland1999}, 
\begin{equation*}
        \sup_x\int\vert K(x,y,x-y)\vert\ \mathrm{d}y+\sup_y\int \vert K(x, y, x-y)\vert\ \mathrm{d}x<\infty.
    \end{equation*}
It follows that $T_K$ is a bounded operator on $L^p$ for $1 \leq p \leq \infty$. As $\varepsilon \to 0$, we have 
\begin{equation*}
        \sup_x\int\vert {\bf 1}_{\{ |x-y| \leq \varepsilon \}} K(x,y,x-y)\vert\ \mathrm{d}y+\sup_y\int \vert {\bf 1}_{\{ |x-y| \leq \varepsilon \}}  K(x, y, x-y)\vert\ \mathrm{d}x = O(\varepsilon^s),
    \end{equation*}
    so Schur's test implies $L^p$ convergence in the second statement. For pointwise convergence a.e., define 
    \[
    T_* f(x) = \sup_{0 < \varepsilon < 1} \left| \int_{B^c(x;\varepsilon)}K(x,y,x-y)f(y)\ \mathrm{d}y \right|.
    \]
    Then 
    \[
    |T_* f(x)| \lesssim \int_{\mathbb{R}^n} \chi(x) \chi(y) |x-y|^{-(n-s)} |f(y)| \,dy,
    \]
    and Schur's test gives that $\|T_* f\|_{L^p} \lesssim \|f\|_{L^p}$. In the last statement, for any $f \in \mathcal{S}$ we have pointwise convergence as $\varepsilon \to 0$, so \cite[Theorem 2.1.14]{grafakos2014} implies pointwise almost everywhere convergence for any fixed $f \in L^p$.
\end{proof}

Next we formulate a result for Calder\'on-Zygmund type operators.

\begin{lemma}\label{Singularintlplimit}
    Suppose $K:\mathbb{R}^n\times\mathbb{R}^n\times(\mathbb{R}^n\setminus\{0\})\to\mathbb{R}$ is such that
    \begin{itemize}
        \item $K$ is of class $C^1$ with respect to the second argument and continuous with respect to the first and third arguments,
        \item $K$ is compactly supported with respect to the first two arguments and
        \item $K$ is positively $(-n)$-homogeneous and odd with respect to the third argument. 
    \end{itemize}
    Then for all $f\in L^p$ the limit
        \begin{equation*}
            (T_Kf)(x)=\mathrm{p.v.}\int K(x,y,x-y)f(y)\ \mathrm{d}y:=\lim_{\varepsilon\to 0}\int_{B^c(x;\varepsilon)} K(x,y,x-y)f(y)\ \mathrm{d}y
        \end{equation*}
        exists in $L^p$ and pointwise almost everywhere. Moreover the operator $T_K$ defined as the above limit is bounded on $L^p$.
\end{lemma}
\begin{proof}
    Let $f\in L^p$ and denote by $\chi=\textbf{1}_{\pi_2(\mathrm{supp}(K))}$. Due to the differentiability assumption on $K$ we can take a Taylor expansion at $y=x$ and write
    \begin{equation*}
        \begin{split}
            (T_{K}^\varepsilon f)(x)&=\int_{B^c(x;\varepsilon)} K(x,x,x-y)\chi(y)f(y)\ \mathrm{d}y\\
            &\quad +\int_{B^c(x;\varepsilon)}\left\langle\int_0^1(\nabla_2K)(x,ty+(1-t)x,x-y)\ \mathrm{d}t, y-x \right\rangle \chi(y)f(y)\ \mathrm{d}y \\
            &=:\int_{B^c(x;\varepsilon)} K_1(x,x-y)\chi(y) f(y)\ \mathrm{d}y+\int_{B^c(x;\varepsilon)} K_2(x,y,x-y)f(y)\ \mathrm{d}y.
        \end{split}
    \end{equation*}
    Since $K_1$ is bounded on $\mathbb{R}^n\times\mathbb{R}^n\times\mathbb{S}^{n-1}$ and positively $(-n)$-homogeneous and odd in the second argument, by \cite{calderon1956singular} we have that the limit
    \begin{equation*}
        \lim_{\varepsilon\to 0} \int_{B^c(x;\varepsilon)} K_1(x,x-y)\chi(y)f(y)\ \mathrm{d}y
    \end{equation*}
    exists in $L^p$ and pointwise almost everywhere. The kernel $K_2$ satisfies the conditions of Lemma \ref{weaksingularpvalue}, giving us the convergence of the second limit.

    To establish boundedness, we note that by \cite{calderon1956singular} we have
    \begin{equation*}
        \Vert \lim_{\varepsilon\to 0}T^\varepsilon_{K_1}\chi f\Vert_{L^p}\lesssim \Vert \chi f\Vert_{L^p}\lesssim \Vert f\Vert_{L^p}
    \end{equation*}
    and by Lemma \ref{weaksingularpvalue} we have
    \begin{equation*}
        \Vert \lim_{\varepsilon\to 0}T^\varepsilon_{K_2}f\Vert_{L^p}\lesssim \Vert f\Vert_{L^p}.
    \end{equation*}
    This gives us the claim.
\end{proof}

Before moving on to regularity estimates, we need the following calculus lemma.
\begin{lemma}\label{movingreg}
    Suppose $f\in C^1_0(\mathbb{R}^n\times\mathbb{R}^n)$. Then for all $\varepsilon>0$ the function
    \begin{equation*}
        x\mapsto \int_{B^c(x;\varepsilon)}f(x,y)\ \mathrm{d}y
    \end{equation*}
    is continuously differentiable with
    \begin{equation*}
        \partial_{i}\int_{B^c(x;\varepsilon)}f(x,y)\ \mathrm{d}y=\int_{B^c(x;\varepsilon)}(\partial_{1,i}f)(x,y)\ \mathrm{d}y-\int_{\partial B(x;\varepsilon)} f(x,y)\frac{(x-y)^i}{\vert x-y\vert}\ \mathrm{d}S(y).
    \end{equation*}
\end{lemma}
\begin{proof}
    We change variables to write
    \begin{equation*}
        \int_{B^c(x;\varepsilon)}f(x,y)\ \mathrm{d}y=\int_{B^c(0;\varepsilon)}f(x,x+z)\ \mathrm{d}z.
    \end{equation*}
    Since the function $f$ is continuously differentiable in the integration domain and compactly supported, we can differentiate under the integral sign and use the divergence theorem to write
    \begin{equation*}
        \begin{split}
            \partial_i\int_{B^c(x;\varepsilon)}f(x,y)\ \mathrm{d}y&=\int_{B^c(0;\varepsilon)}(\partial_{1,i}f)(x,x+z)\ \mathrm{d}z+\int_{B^c(0;\varepsilon)}\partial_{z^i}f(x,x+z)\ \mathrm{d}z\\
            &=\int_{B^c(0;\varepsilon)}(\partial_{1,i}f)(x,x+z)\ \mathrm{d}z-\int_{\partial B^c(0;\varepsilon)}f(x,x+z)\ \frac{z^i}{\vert z\vert}\ \mathrm{d}S(z) \\
            &=\int_{B^c(x;\varepsilon)}(\partial_{1,i}f)(x,y)\ \mathrm{d}y-\int_{\partial B(x;\varepsilon)} f(x,y)\frac{(x-y)^i}{\vert x-y\vert}\ \mathrm{d}S(y). \qedhere
        \end{split}
    \end{equation*}
\end{proof}

We need two regularity results, one for kernels with integrable derivatives and one for kernels with singular derivatives.

\begin{lemma}\label{weaklysingulardifferentiability}
    Let $s>1$ and suppose $K:\mathbb{R}^n\times\mathbb{R}^n\times(\mathbb{R}^n\setminus\{0\})\to\mathbb{R}$ is
    \begin{itemize}
        \item of class $C^1$ and compactly supported in the first two arguments and
        \item smooth and positively $-(n-s)$-homogeneous in the last argument.
    \end{itemize}
    Then the associated integral operator, defined on $\mathcal{S}$ by
    \begin{equation*}
        (T_K\phi)(x)=\int K(x,y,x-y)\phi(y)\ \mathrm{d}y,
    \end{equation*}
    extends uniquely to a bounded operator $T_K:L^p\to W^{1,p}$.
\end{lemma}
\begin{proof}
    Since
    \begin{equation}
        \sup_x\int\vert (\partial_{1,i}K+\partial_{3,i}K)(x,y,x-y)\vert\ \mathrm{d}y+\sup_y\int\vert (\partial_{1,i}K+\partial_{3,i}K)(x,y,x-y)\vert\ \mathrm{d}x<\infty,
    \end{equation}
    by Schur's test the operators $T_{\partial_{1,i}K+\partial_{3,i}K}$, defined on $\mathcal{S}$ by
    \begin{equation*}
        (T_{\partial_{1,i}K+\partial_{3,i}K}\phi)(x)=\int (\partial_{1,i}K+\partial_{3,i}K)(x,y,x-y)\phi(y)\ \mathrm{d}y,
    \end{equation*}
    extend to bounded operators $T_{\partial_{1,i}K+\partial_{3,i}K}:L^p\to L^p$. Moreover by Lemma \ref{weaksingularpvalue} we have for $f\in L^p$ that
    \begin{equation*}
        (T_{\partial_{1,i}K+\partial_{3,i}K}f)(x)=\lim_{\varepsilon\to 0}\int_{B^c(x;\varepsilon)}(\partial_{1,i}K+\partial_{3,i}K)(x,y,x-y)f(y)\ \mathrm{d}y
    \end{equation*}
    in $L^p$. 

    We claim that $\partial_iT_Kf=T_{\partial_{1,i}K+\partial_{3,i}K}f$ for $f\in L^p$. To show this, let $\phi\in C^\infty_0$ and suppose $f\in C^\infty_0$. Then by the $L^p$-continuity of the functional $\langle \phi,\bullet\rangle$ we have
    \begin{equation*}
        \begin{split}
            &\int\phi(x)\left(\int (\partial_{1,i}K+\partial_{3,i}K)(x,y,x-y)f(y)\ \mathrm{d}y \right)\ \mathrm{d}x\\
            &=\lim_{\varepsilon\to 0}\int\phi(x)\left(\int_{B^c(x;\varepsilon)} (\partial_{1,i}K+\partial_{3,i}K)(x,y,x-y)f(y)\ \mathrm{d}y \right)\ \mathrm{d}x.
        \end{split}
    \end{equation*}
    By Lemma $\ref{movingreg}$ we have
    \begin{equation*}
        \begin{split}
            \int_{B^c(x;\varepsilon)} (\partial_{1,i}K+\partial_{3,i}K)(x,y,x-y)f(y)\ \mathrm{d}y&=\partial_i\int_{B^c(x;\varepsilon)} K(x,y,x-y)f(y)\ \mathrm{d}y\\
            &\quad +\int_{\partial B(x;\varepsilon)}K(x,y,x-y)f(y)\frac{(x-y)^i}{\vert x-y\vert}\ \mathrm{d}S(y).
        \end{split}
    \end{equation*}
    Since
    \begin{equation*}
        \left\vert\int_{\partial B(x;\varepsilon)}K(x,y,x-y)f(y)\frac{(x-y)^i}{\vert x-y\vert}\ \mathrm{d}S(y)\right\vert\lesssim \frac{\varepsilon^{n-1}}{\varepsilon^{n-s}}=\varepsilon^{s-1}\to 0
    \end{equation*}
    as $\varepsilon\to  0$, we may use for $\varepsilon$ small enough the integrability of the dominating function
    \begin{equation*}
        x\mapsto \vert \phi(x)\vert
    \end{equation*}
    to apply the dominated convergence theorem and conclude that
    \begin{equation*}
        \lim_{\varepsilon\to 0}\int \phi(x)\left(\int_{\partial B(x;\varepsilon)}K(x,y,x-y)f(y)\frac{(x-y)^i}{\vert x-y\vert}\ \mathrm{d}S(y)\right)\ \mathrm{d}x=0.
    \end{equation*}
    With this we have
    \begin{equation*}
        \begin{split}
            &\lim_{\varepsilon\to 0}\int\phi(x)\left(\int_{B^c(x;\varepsilon)} (\partial_{1,i}K+\partial_{3,i}K)(x,y,x-y)f(y)\ \mathrm{d}y \right)\ \mathrm{d}x \\
            &=\lim_{\varepsilon\to 0}\int\phi(x)\partial_i\left(\int_{B^c(x;\varepsilon)} K(x,y,x-y)f(y)\ \mathrm{d}y \right)\ \mathrm{d}x \\
            &=-\lim_{\varepsilon\to 0}\int(\partial_i\phi)(x)\left(\int_{B^c(x;\varepsilon)} K(x,y,x-y)f(y)\ \mathrm{d}y \right)\ \mathrm{d}x \\
            &=-\int(\partial_i\phi)(x)\left(\int K(x,y,x-y)f(y)\ \mathrm{d}y \right)\ \mathrm{d}x,
        \end{split}
    \end{equation*}
    where the last equality follows from the $L^p$-continuity of the functional $\langle\partial_i\phi,\bullet\rangle$ and the $L^p$-convergence of the limit
    \begin{equation*}
        \lim_{\varepsilon\to 0}\int _{B^c(x;\varepsilon)}K(x,y,x-y)f(y)\ \mathrm{d}y. \qedhere
    \end{equation*}
\end{proof}

For kernels with singular derivatives we have the following.

\begin{lemma}\label{stronglysingulardifferentiability}
    Suppose $K:\mathbb{R}^n\times\mathbb{R}^n\times(\mathbb{R}^n\setminus\{0\})\to\mathbb{R}$ is
    \begin{itemize}
        \item of class $C^1$ and compactly supported in the first two arguments and
        \item smooth, even and positively $(-(n-1))$-homogeneous in the last argument.
    \end{itemize}
    Then the associated integral operator, defined on $\mathcal{S}$ by the absolutely convergent integral
        \begin{equation*}
            (T_K\phi)(x)=\int K(x,y,x-y)\phi(y)\ \mathrm{d}y,
        \end{equation*}
        extends uniquely to a bounded operator $T_K:L^p\to W^{1,p}$.
\end{lemma}
\begin{proof}
   The $L^p$-boundedness follows by Schur's test as in the proof of Lemma \ref{weaklysingulardifferentiability}.

   To show differentiability, we claim that
   \begin{equation*}
       (\partial_iT_Kf)(x)=\lim_{\varepsilon\to 0}\int_{B^c(x;\varepsilon)}(\partial_{1,i}K+\partial_{3,i}K)(x,y,x-y)f(y)\ \mathrm{d}y,
   \end{equation*}
   where the limit exists in $L^p$, since the kernels $\partial_{1,i}K$ satisfy the conditions of Lemma \ref{weaksingularpvalue} and the kernels $\partial_{3,i}K$ satisfy the conditions of Lemma \ref{Singularintlplimit}. Let $\phi\in C^\infty_0$ and suppose $f\in C^\infty_0$. The argument to show that
   \begin{equation*}
       \int\phi(x)\left(\lim_{\varepsilon\to 0}\int_{B^c(x;\varepsilon)}(\partial_{1,i}K+\partial_{3,i}K)(x,y,x-y)f(y)\ \mathrm{d}y\right)\ \mathrm{d}x \\
       =-\int(\partial_i\phi)(x)\left(\int K(x,y,x-y)f(y)\ \mathrm{d}y\right)\ \mathrm{d}x
   \end{equation*}
   is the same used in the proof of Lemma \ref{weaklysingulardifferentiability}, except that the vanishing of the boundary term has to be argued differently. To do this, we use the smoothness assumption on the kernel and take a Taylor expansion at $z=0$ to write
   \begin{equation*}
       \begin{split}
           &\int_{\partial B(0;\varepsilon)}K(x,x-z,z)f(x-z)\frac{z^i}{\vert z\vert}\ \mathrm{d}S(z) \\
           &=\int_{\partial B(0;\varepsilon)}K(x,x,z)f(x)\frac{z^i}{\vert z\vert}\ \mathrm{d}S(z)+\int_{\partial B(0;\varepsilon)}r(x,z)\frac{z^i}{\vert z\vert}\ \mathrm{d}S(z),
       \end{split}
   \end{equation*}
   where the function $r$ satisfies the estimate
   \begin{equation*}
       \vert r(x,z)\vert=\mathcal{O}\left(\frac{1}{\vert z\vert^{n-2}}\right)
   \end{equation*}
   as $z\to 0$. Since the kernel $K$ is even with respect to the last argument, we have
   \begin{equation*}
       \int_{\partial B(0;\varepsilon)}K(x,x,z)f(x)\frac{z^i}{\vert z\vert}\ \mathrm{d}S(z)=0.
   \end{equation*}
   Since 
   \begin{equation*}
       \left\vert\int_{\partial B(0;\varepsilon)}r(x,z)\frac{z^i}{\vert z\vert}\ \mathrm{d}S(z)\right\vert\lesssim \frac{\varepsilon^{n-1}}{\varepsilon^{n-2}}=\varepsilon,
   \end{equation*}
   we can now use the argument in the proof of Lemma \ref{weaklysingulardifferentiability} to conclude the vanishing of the boundary term.
\end{proof}

We can now prove the base case $l=0$ for Proposition \ref{Highermapping}.

\begin{lemma}\label{perustapaus}
    Suppose $K:\mathbb{R}^n\times\mathbb{R}^n\times(\mathbb{R}^n\setminus\{0\})\to \mathbb{R}$ is
    \begin{itemize}
        \item of class $C^{2}$ and compactly supported in the first two arguments and
        \item smooth, odd and positively $-(n-2)$-homogeneous in the last argument.
    \end{itemize}
    Then the associated integral operator $T_K$, defined on $\mathcal{S}$ by the uniformly convergent integral
    \begin{equation*}
        (T_K\phi)(x)=\int K(x,y,x-y)\phi(y)\ \mathrm{d}y,
    \end{equation*}
    extends uniquely to a bounded operator $T_K:L^p\to W^{2,p}$.
\end{lemma}

\begin{proof}
    By Lemma \ref{weaklysingulardifferentiability} we have $T_K:L^p\to W^{1,p}$ with
    \begin{equation*}
        (\partial_jT_Kf)(x)=\int (\partial_{1,j}K+\partial_{3,j}K)(x,y,x-y)f(y)\ \mathrm{d}y.
    \end{equation*}
    Since the kernels $\partial_{1,j}K$ satisfy the conditions of Lemma \ref{weaklysingulardifferentiability} and the kernels $\partial_{3,j}K$ satisfy the conditions of Lemma \ref{stronglysingulardifferentiability}, we have that $\partial_jT_K:L^p\to W^{1,p}$ with
    \begin{equation*}
        \begin{split}
            (\partial_i\partial_jT_Kf)(x)&=\int(\partial_{1,i}\partial_{1,j}K+\partial_{3,i}\partial_{1,j}K)(x,y,x-y)f(y)\ \mathrm{d}y \\
            &\quad +\int(\partial_{1,i}\partial_{3,j}K+\partial_{3,i}\partial_{3,j}K)(x,y,x-y)f(y)\ \mathrm{d}y,
        \end{split}
    \end{equation*}
    giving us the base case.
\end{proof}

Next we prove a useful commutator result used to treat higher order derivatives.

\begin{lemma}\label{commu}
    Let \(K(x,y,z)\) be a kernel as in Lemma \ref{perustapaus} except that $K$ is $C^3$ in $x$ and $y$, and let \(f\in W^{1,p}\). Then $[\partial_j,T_K]:L^p\to W^{2,p}$ and
    \begin{equation*}
        ([\partial_j,T_K]f)(x)=\int
        \left(\partial_{1,j}K(x,y,x-y)+\partial_{2,j}K(x,y,x-y)\right)f(y)\mathrm{d}y.
    \end{equation*}
\end{lemma}

\begin{proof}
    By Lemma \ref{weaklysingulardifferentiability} we have that
    \begin{equation*}
        \partial_jT_Kf(x)=\int(\partial_{1,j}K+\partial_{3,j}K)(x,y,x-y)f(y)\mathrm{d}y.
    \end{equation*}

    Next we use integration by parts to obtain
    \begin{equation*}
    \begin{split}
        T_K(\partial_jf)(x)&=\lim_{\varepsilon\to0}\int_{B^c(x,\varepsilon)} K(x,y,x-y)\partial_{y^j}f(y)\mathrm{d}y\\
        &=-\lim_{\varepsilon\to0}\int_{B^c(x,\varepsilon)}\partial_{y^j}(K(x,y,x-y))f(y)\mathrm{d}y+\lim_{\varepsilon\to0}\int_{\partial B(x,\varepsilon)}K(x,y,x-y)f(y)\frac{y^j-x^j}{|y-x|}dS(y)\\
        &=-\lim_{\varepsilon\to0}\int_{B^c(x,\varepsilon)} (\partial_{2,j}K-\partial_{3,j}K)(x,y,x-y)f(y)\mathrm{d}y+\lim_{\varepsilon\to0}\int_{\partial B(x,\varepsilon)}K(x,y,x-y)f(y)\frac{y^j-x^j}{|y-x|}dS(y)\\
        &=-\int(\partial_{2,j}K-\partial_{3,j}K)(x,y,x-y)f(y)\mathrm{d}y,\\
    \end{split}
    \end{equation*}
    where we used the fact that boundary term is zero, which was already shown in the proof of Lemma \ref{weaklysingulardifferentiability}.

    Hence
    \begin{equation*}
    \begin{split}
        ([\partial_j,T_K]f)(x)&=\int(\partial_{1,j}K+\partial_{3,j}K+\partial_{2,j}K-\partial_{3,j}K)(x,y,x-y)f(y)\mathrm{d}y.\\
    \end{split}
    \end{equation*}
    Since both \(\partial_{1,j}K\) and \(\partial_{2,j}K\) satisfy properties of Lemma \ref{perustapaus}, we have that \([\partial_j,T_K]:L^p\to W^{2,p}\).
\end{proof}

For higher derivatives we have the following Lemma.

\begin{lemma}\label{higherderivativeformula}
    Suppose $K$ satisfies the conditions of Lemma \ref{perustapaus} and is in addition of class $C^k$ with respect to the first two arguments for some $k\ge 3$. Then for $f\in W^{k,p}$ we have $T_Kf\in W^{k,p}$, and for all $\alpha$ with $\vert \alpha\vert\leq k$ we have 
    \begin{equation*}
        \partial^\alpha T_Kf=\sum_{\beta\leq \alpha}\sum_{\eta\leq \beta}\binom{\alpha}{\beta}\binom{\beta}{\eta}T_{\partial_1^\eta \partial_2^{\beta-\eta}K}\partial^{\alpha-\beta}f.
    \end{equation*}
\end{lemma}
\begin{proof}
We begin by showing that
\begin{equation*}
    \begin{split}
        \partial^\beta_x (K(x,x-z,z))=\sum_{\eta\leq \beta}\binom{\beta}{\eta}(\partial_1^\eta\partial_2^{\beta-\eta}K)(x,x-z,z)
    \end{split}
\end{equation*}
for $\beta$ with $\vert\beta\vert\leq k$.

If $\vert\beta\vert=1$, then we get
\begin{equation*}
    \partial_{x^i}(K(x,x-z,z))=(\partial_{1,i}K+\partial_{2,i}K)(x,x-z,z),
\end{equation*}
giving us the base case.

Suppose then that the claim holds for all $\beta$ with $\vert\beta\vert\leq k-1$. Then
\begin{equation*}
    \begin{split}
    \partial_{x^i}\partial_x^\beta (K(x,x-z,z)) &=\sum_{\eta\leq \beta}\binom{\beta}{\eta}(\partial_1^{\eta+e_i}\partial_2^{\beta-\eta}K+\partial_1^\eta\partial_2^{\beta+e_i-\eta}K)(x,x-z,z) \\
        &=\sum_{j=0}^{\beta_i}\sum_{\substack{\eta\leq \beta \\
        \eta_i=j}}\binom{\beta_i}{j}\prod_{k\neq i}\binom{\beta_k}{\eta_k}(\partial_1^{\eta+e_i}\partial_2^{\beta-\eta}K)(x,x-z,z) \\
        &\quad +\sum_{j=0}^{\beta_i}\sum_{\substack{\eta\leq \beta \\
        \eta_i=j}}\binom{\beta_i}{j}\prod_{k\neq i}\binom{\beta_k}{\eta_k}(\partial_1^{\eta}\partial_2^{\beta+e_i-\eta}K)(x,x-z,z).
    \end{split}
\end{equation*}

We then change the summation index in the first term and separate the $j=0$ summand in the second term to write the above as
\begin{equation*}
    \begin{split}
        &\sum_{j=1}^{\beta_i+1}\sum_{\substack{\eta\leq \beta \\
        \eta_i=j}}\binom{\beta_i}{j-1}\prod_{k\neq i}\binom{\beta_k}{\eta_k}(\partial_1^{\eta-e_i+e_i}\partial_2^{\beta-(\eta-e_i)}K)(x,x-z,z) \\
        &\quad +\sum_{\substack{\eta\leq \beta \\
        \eta_i=0}}\binom{\beta_i}{0}\prod_{k\neq i}\binom{\beta_k}{\eta_k}(\partial_1^{\eta}\partial_2^{\beta+e_i-\eta}K)(x,x-z,z) \\
        &\quad +\sum_{j=1}^{\beta_i}\sum_{\substack{\eta\leq \beta \\
        \eta_i=j}}\binom{\beta_i}{j}\prod_{k\neq i}\binom{\beta_k}{\eta_k}(\partial_1^{\eta}\partial_2^{\beta+e_i-\eta}K)(x,x-z,z) \\
        &=\sum_{j=1}^{\beta_i}\sum_{\substack{\eta\leq \beta \\
        \eta_i=j}}\left(\binom{\beta_i}{j}+\binom{\beta_i}{j-1}\right)\prod_{k\neq i}\binom{\beta_k}{\eta_k}(\partial_1^{\eta}\partial_2^{\beta+e_i-\eta}K)(x,x-z,z) \\
        &\quad +\sum_{\substack{\eta\leq \beta \\
        \eta_i=0}}\binom{\beta_i}{0}\prod_{k\neq i}\binom{\beta_k}{\eta_k}(\partial_1^{\eta}\partial_2^{\beta+e_i-\eta}K)(x,x-z,z) \\
        &\quad +\sum_{\substack{\eta\leq \beta \\
        \eta_i=\beta_i+1}}\binom{\beta_i}{\beta_i}\prod_{k\neq i}\binom{\beta_k}{\eta_k}(\partial_1^{\eta}\partial_2^{\beta+e_i-\eta}K)(x,x-z,z).
    \end{split}
\end{equation*}
Finally, we use the identity
\begin{equation*}
    \binom{\beta_i}{j}+\binom{\beta_i}{j-1}=\binom{\beta_i+1}{j}
\end{equation*}
to write the above as
\begin{equation*}
    \begin{split}
    &\sum_{j=1}^{\beta_i}\sum_{\substack{\eta\leq \beta \\
        \eta_i=j}}\binom{\beta_i+1}{j}\prod_{k\neq i}\binom{\beta_k}{\eta_k}(\partial_1^{\eta}\partial_2^{\beta+e_i-\eta}K)(x,x-z,z) \\
        &\quad +\sum_{\substack{\eta\leq \beta \\
        \eta_i=0}}\binom{\beta_i+1}{0}\prod_{k\neq i}\binom{\beta_k}{\eta_k}(\partial_1^{\eta}\partial_2^{\beta+e_i-\eta}K)(x,x-z,z) \\
        &\quad +\sum_{\substack{\eta\leq \beta \\
        \eta_i=\beta_i+1}}\binom{\beta_i+1}{\beta_i+1}\prod_{k\neq i}\binom{\beta_k}{\eta_k}(\partial_1^{\eta}\partial_2^{\beta+e_i-\eta}K)(x,x-z,z) \\
        &=\sum_{j=0}^{\beta_i+1}\sum_{\substack{\eta\leq \beta \\
        \eta_i=j}}\binom{\beta_i+1}{j}\prod_{k\neq i}\binom{\beta_k}{\eta_k}(\partial_1^{\eta}\partial_2^{\beta+e_i-\eta}K)(x,x-z,z) \\
        &=\sum_{\eta\leq \beta+e_i}\binom{\beta+e_i}{\eta}(\partial_1^\eta\partial_2^{\beta-\eta}K)(x,x-z,z),
    \end{split}
\end{equation*}
giving us the claim by induction.

We then move on to show that 
    \begin{equation*}
        (\partial^\alpha T_Kf)(x)=\int \partial_x^\alpha(K(x,x-z,z)f(x-z))\ \mathrm{d}z
    \end{equation*}
    for all $\alpha$ with $\vert \alpha\vert\leq k$.
    
    By Lemma \ref{commu} and the change of variables to $z=x-y$ we have
    \begin{equation*}
        \begin{split}
            (\partial_iT_Kf)(x)&=\int\left(\left(\partial_{1,i}K+\partial_{2,i}K\right)(x,y,x-y)f(y)+K(x,y,x-y)\partial_if(y)\right)\ \mathrm{d}y \\
            &=\int \partial_{x^i}(K(x,x-z,z)f(x-z))\ \mathrm{d}z,
        \end{split}
    \end{equation*}
    giving us the case $\vert \alpha\vert=1$.

    Suppose then that $k\ge 2$ and that the claim holds for $\alpha$ with $\vert\alpha\vert=k-1$. The kernels
    \begin{equation*}
        \partial_1^\eta\partial_2^{\beta-\eta}K
    \end{equation*}
    and the functions $\partial^{\alpha-\beta}f$ satisfy the conditions of Lemma \ref{commu} for each $\beta\leq\alpha$ and $\eta\leq \beta$. Hence by the Leibniz rule and Lemma \ref{commu} we get
    \begin{equation*}
        \begin{split}
            (\partial_i\partial^\alpha T_Kf)&=\sum_{\beta\leq \alpha}\sum_{\eta\leq \beta}\binom{\alpha}{\beta}\binom{\beta}{\eta}\partial_iT_{\partial_1^\eta\partial_2^{\beta-\eta}K}\partial^{\alpha-\beta}f \\
            &=\sum_{\eta\leq \beta}\binom{\alpha}{\beta}\binom{\beta}{\eta}\int\partial_{x^i}(\partial_1^\eta\partial_2^{\beta-\eta}K(x,x-z,z)\partial^{\alpha-\beta}f(x-z))\ \mathrm{d}z \\
            &=\int\partial_x^{\alpha+e_i}(K(x,x-z,z)f(x-z))\ \mathrm{d}z,
        \end{split}
    \end{equation*}
    giving us the claim by induction.

    By expanding the derivatives out we then get that
    \begin{equation*}
        \begin{split}
            (\partial^\alpha T_Kf)(x)&=\sum_{\beta\leq \alpha}\sum_{\eta\leq \beta}\binom{\alpha}{\beta}\binom{\beta}{\eta}(T_{\partial_1^\eta\partial_2^{\beta-\eta}K}\partial^{\alpha-\beta}f)(x)
        \end{split}
    \end{equation*}
    for $\alpha$ with $\vert\alpha\vert\leq k$.
\end{proof}

We are now in a position to prove Proposition \ref{Highermapping}.

\begin{proof}[Proof of Proposition \ref{Highermapping}]
    The base case $l=0$ was shown in Lemma \ref{perustapaus}.
    
    Let then $\vert\alpha\vert=l$ and $f\in W^{l,p}$. Then by Lemma \ref{higherderivativeformula} we have
    \begin{equation*}
        \partial^\alpha T_Kf=\sum_{\beta\leq \alpha}\sum_{\eta\leq \beta}\binom{\alpha}{\beta}\binom{\beta}{\eta}T_{\partial_1^\eta\partial_2^{\beta-\eta}K}\partial^{\alpha-\beta}f.
    \end{equation*}
    Since all the kernels $\partial_1^\eta\partial_2^{\beta-\eta}K$ are as in the base case, we have that $\partial^\alpha T_K f\in W^{2,p}$, giving us the claim.
\end{proof}

We then get the following regularity result for the remainder term $R$ in Lemma \ref{decomposition}.

\begin{corollary}\label{rmappingproperties}
    Suppose $g\in C^k$ with $k\ge 5$. Then $R:W^{l,p}\to W^{l+2,p}$ for $l\leq k-5$.
\end{corollary}

\begin{proof}
    If $g\in C^k$ with $k\ge 5$, then Lemma \ref{arrankerneli} lets us apply Proposition \ref{Highermapping} to conclude the claim.
\end{proof}

\subsection{Mapping properties for the remainder terms: Hölder estimates}

We move on to establish the Hölder regularity required for Theorem \ref{thm_main2}.

\begin{lemma}\label{alpha_mapping}
    Suppose $K:\mathbb{R}^n\times\mathbb{R}^n\times(\mathbb{R}^n\setminus\{0\})\to\mathbb{R}$ is
    \begin{itemize}
        \item of class \(C^\alpha\) and compactly supported in the first two arguments and
        \item \(K\) is smooth and positively $-(n-s)$-homogeneous in the last argument, where $s>0$.
    \end{itemize}
    Then \(T_K\) maps \(C^\alpha_0\) to \(C^\alpha\).
\end{lemma}

\begin{proof}
    Let \(f\in C^\alpha_0\). Now we can write
    \begin{equation*}
        T_Kf(x)=\int K(x,y,x-y)f(y)\ \mathrm{d}y=\int K(x,x-z,z)f(x-z)\ \mathrm{d}z.
    \end{equation*}
    Since \(K\) is compactly supported in the first two variables and $f$ is compactly supported, we have 
    \begin{equation*}
    \begin{split}
        |T_Kf(x)-T_Kf(y)|&=\left\vert\int\left(K(x,x-z,z)f(x-z)-K(y,y-z,z)f(y-z)\right)\ \mathrm{d}z\right\vert\\
        &\leq \int |K(x,x-z,z)|\left|f(x-z)-f(y-z)\right|\ \mathrm{d}z \\
        &\quad +\int |\chi(z)|\left|K(x,x-z,z)-K(y,y-z,z)\right||f(z)|\ \mathrm{d}z
    \end{split}
    \end{equation*}
    for some $\chi \in C^{\infty}_0$. For the first term we can use the fact that \(f\in C^\alpha\) and therefore
    \begin{equation*}
        \int |K(x,x-z,z)|\left|f(x-z)-f(y-z)\right|\ \mathrm{d}z\leq C\sup_x\int|K(x,x-z,z)|\ \mathrm{d}z|x-y|^\alpha.
    \end{equation*}
    For the second term we write \(K(x,x-z,z)=\frac{K(x,x-z,\hat{z})}{|z|^{n-s}}\). Then
    \begin{equation*}
    \begin{split}
            \left|K(x,x-z,z)-K(y,y-z,z)\right|&\leq \frac{1}{|z|^{n-s}}\left(\left|K(x,x-z,\hat{z})-K(x,y-z,\hat{z})\right|+\left|K(x,y-z,\hat{z})-K(y,y-z,\hat{z})\right|\right)\\
            &\lesssim \frac{1}{|z|^{n-s}}|x-y|^\alpha
    \end{split}
    \end{equation*}
    when $x, y, z$ are in a fixed compact set. Then, for $x, y$ in a compact set  
    \begin{equation*}
            \int |\chi(z)|\left|K(x,x-z,z)-K(y,y-z,z)\right||f(z)|\ \mathrm{d}z\\
            \lesssim |x-y|^\alpha \int\frac{|\chi(z)|}{|z|^{n-s}} \ \mathrm{d}z \lesssim|x-y|^\alpha. \qedhere
    \end{equation*}
\end{proof}

\begin{lemma}\label{höldermapping}
    Let \(g\in C^{k,\alpha}\) with \(k\geq5\) and let \(f\in C^\alpha\) satisfy \(If=0\). Then \(Rf\in C^{2,\alpha}\).
\end{lemma}

\begin{proof}
    The proof is based on Lemma \ref{commu} and Sobolev embedding. First we note that \(f\) is compactly supported in \(\mathbb{R}^n\) and therefore \(f\in L^p\) for any \(p\in(1,\infty)\). Hence \(Rf\in W^{2,p}\), where \(p>\frac{n}{2}\). By Lemma \ref{commu} we may write
    \begin{equation*}
        \partial_jRf=R\partial_jf+[\partial_j,R]f.
    \end{equation*}
    By Proposition \ref{Highermapping} and Lemma \ref{weaklysingulardifferentiability} we have that \(R\partial_jf\in W^{2,p}\) and \([\partial_j,R]f\in W^{1,p}\). By Sobolev embedding we have that \(\partial_jRf\in C^\alpha\). Therefore the first derivative of $Rf$ is continuous and compactly supported. For the second derivative we use again the commutator result to write
    \begin{equation*}
        \partial_i\partial_jRf=\partial_iR\partial_jf+\partial_i[\partial_j,R]f=\partial_iR\partial_jf+[\partial_j,R]\partial_if+[\partial_i,[\partial_j,R]]f.
    \end{equation*}
    Now \(\partial_iR\partial_jf, [\partial_j,R]\partial_if\in W^{1,p}\) and by Lemma \ref{alpha_mapping} we have that \([\partial_i,[\partial_j,R]]f\in C^\alpha\). Hence \(\partial_i\partial_jRf\in C^\alpha\), so the second derivative is also continuous and compactly supported. By \cite[Theorem 3.1.7]{zbMATH01950198} we have that \(Rf\in C^{2,\alpha}\).
\end{proof}

\appendix 

\section{Properties of Geodesics}\label{sect:propofgeodesics}

In this appendix we verify that certain basic properties of geodesics on smooth Riemannian manifolds (see e.g.\ \cite{lee2018introduction, petersen2006riemannian}) remain valid when the metric $g$ has low regularity.

Let $\Omega \subset \mathbb{R}^n$ be open and let $g = (g_{jk})$ be a $C^{1,1}$ Riemannian metric in $\Omega$, i.e.\ each $g_{jk}$ is a $C^{1,1}$ function in $\Omega$ with $g_{jk} = g_{kj}$ and $g_{jk}(x) v^j v^k \geq c |v|_{\mathrm{Eucl}}^2$ for some $c > 0$ uniformly over $x \in \Omega, v \in \mathbb{R}^n$.
By Cramer's formula, the inverse matrix $(g^{jk})$ is also $C^{1,1}$. 
The metric has Christoffel symbols 
\[
\Gamma_{jk}^l = \frac{1}{2} g^{lm}(\partial_j g_{km} + \partial_k g_{jm} - \partial_m g_{jk}) \in C^{0,1}(\Omega).
\]
For any $x \in \Omega$ and any $v \in T_x \Omega = \mathbb{R}^n$, we consider the geodesic $\gamma(t) = \gamma_{x,v}(t)$ solving the geodesic equation 
\[
\ddot{\gamma}^l(t) + \Gamma_{jk}^l(\gamma(t)) \dot{\gamma}^j(t) \dot{\gamma}^k(t) = 0, \qquad \gamma(0) = x, \ \ \dot{\gamma}(0) = v.
\]
More precisely, writing $z(t) = z_{x,v}(t) = (\gamma(t), \eta(t))$ where $\eta(t) = \dot{\gamma}_{x,v}(t)$,  the geodesic equation is equivalent with 
\[
\dot{z}(t) = H(z(t)), \qquad z(0) = (x,v),
\]
where $H(\gamma, \eta) = (\eta, -\Gamma_{jk}^l(\gamma) \eta^j \eta^k)$. Since $H$ is Lipschitz, the existence and uniqueness theorem for ODEs shows that there is a unique solution $z: I_{x,v} \to \Omega$ in some maximal open interval $I_{x,v}$ containing $0$. Moreover, if $K \subset T \Omega$ is a compact set, the lower bound for the maximal interval (see e.g.\ \cite[Section 1.2]{taylor2011partial}) shows that there is $\varepsilon = \varepsilon_K > 0$ and an open set $U = U_K$ containing $K$ such that $[-\varepsilon, \varepsilon] \subset I_{x,v}$ for $(x, v) \in U$.

Up to now we have only assumed that $g$ is $C^{1,1}$. Next we discuss the regularity of geodesics when $g$ is $C^k$ or $C^{k,\alpha}$ with $k \geq 2$ and $0 < \alpha < 1$. We begin with the following lemma.

\begin{lemma} \label{lemma_flow_regularity}
    Let $\Omega\subset\mathbb{R}^n$ be an open set and let $F:\Omega\to\mathbb{R}^n$ belong to $C^k(\Bar{\Omega})$ (resp.\ $C^{k,\alpha}(\Bar{\Omega})$) for some $k\ge 1$ and $0 < \alpha < 1$. Denote by $\Phi$ the flow map, defined by
    \begin{equation*}
        \partial_t\Phi(t,x) =F(\Phi(t,x)), \qquad \Phi(0,x) = x,
    \end{equation*}
    defined in a maximal flow domain $D \subset \mathbb{R} \times \Omega$ with $\{0\} \times \Omega \subset D$. Then $\Phi$ is $C^k$ (resp.\ $C^{k,\alpha}$) in $D$.
\end{lemma}
\begin{proof}
    If $F\in C^k(\Omega)$, then by \cite[Section 1.6]{taylor2011partial} the flow map $\Phi$ is $C^k$.
    
    Next assume that $F\in C^{k,\alpha}(\Omega)$. By the above $\Phi$ is $C^k$, and by the chain rule the functions $\partial_x^\gamma\Phi$ solve the ODEs
\begin{equation}\label{odeforderivatives}
    \partial_t\partial_x^\gamma\Phi(t,x)=F_{,j}(\Phi(t,x))\partial^\gamma\Phi^j(t,x)+\beta_\gamma(t,\Phi(t,x))=:A(t,\Phi(t,x))\partial_x^\gamma\Phi(t,x)+\beta(t,\Phi(t,x)),
\end{equation}
where the functions $A(t,\bullet)$ and $\beta(t,\bullet)$ are $C^\alpha$. 

Since the flow map is Lipschitz, by the uniqueness of solutions we have
\begin{equation*}
    \frac{\vert h(t,\Phi(t,x))-h(t,\Phi(t,y))\vert}{\vert x-y\vert^{\alpha}}\lesssim_t \frac{\vert h(t,\Phi(t,x))-h(t,\Phi(t,y))\vert}{\vert \Phi(t,x)-\Phi(t,y)\vert^{\alpha}}\leq \Vert h(t,\bullet)\Vert_{C^\alpha(\Omega)}
\end{equation*}
for all $h$ such that $h(t,\bullet)\in C^\alpha(\Omega)$. Hence the functions $x\mapsto A(t,\Phi(t,x))$ and $x\mapsto \beta(t,\Phi(t,x))$ are in $C^\alpha(\Omega)$. 

Denote $A(t,\Phi(t,x))=\Tilde{A}(t,x)$ and $\beta(t,\Phi(t,x))=\Tilde{\beta}(t,x)$. Since $\partial_x^{\gamma}\Phi(0,x)=:\partial_x^{\gamma}\Phi(0)$ is constant for $k\ge 1$, we can solve the equation \eqref{odeforderivatives} to get the expression
\begin{equation*}
    \begin{split}
        \partial_{x}^{\gamma}\Phi(t,x)-\partial_x^{\gamma}\Phi(t,y)&=\left(\exp\left(\int_0^t\Tilde{A}(\tau,x)\ \mathrm{d}\tau\right)-\exp\left(\int_0^t\Tilde{A}(\tau,y)\ \mathrm{d}\tau\right)\right)\partial_x^{\gamma}\Phi(0) \\
        &\quad +\int_0^t\exp\left(\int_0^{t-s}\Tilde{A}(\tau,x)\ \mathrm{d}\tau\right)\Tilde{\beta}(s,x)\ \mathrm{d}s-\int_0^t\exp\left(\int_0^{t-s}\Tilde{A}(\tau,y)\ \mathrm{d}\tau\right)\Tilde{\beta}(s,y)\ \mathrm{d}s \\
        &=\left(\exp\left(\int_0^t\Tilde{A}(\tau,x)\ \mathrm{d}\tau\right)-\exp\left(\int_0^t\Tilde{A}(\tau,y)\ \mathrm{d}\tau\right)\right)\partial_x^{\gamma}\Phi(0) \\
        &\quad +\int_0^t\left(\exp\left(\int_0^{t-s}\Tilde{A}(\tau,x)\ \mathrm{d}\tau\right)-\exp\left(\int_0^{s-t}\Tilde{A}(\tau,y)\ \mathrm{d}\tau\right)\right)\beta(s,x)\ \mathrm{d}s \\
        &\quad +\int_0^t\exp\left(\int_0^{t-s}\Tilde{A}(\tau,y)\ \mathrm{d}\tau\right)(\Tilde{\beta}(s,x)-\Tilde{\beta}(s,y))\ \mathrm{d}s.
    \end{split}
\end{equation*}
Now, since $\Tilde{A}(t,\bullet)$ and $\Tilde{\beta}(t,\bullet)$ are $C^\alpha$ and the exponential function is locally Lipschitz, we may estimate
\begin{equation*}
    \begin{split}
        \vert \partial_{x}^{\gamma}\Phi(t,x)-\partial_x^{\gamma}\Phi(t,y)\vert&\lesssim_t\sup_{\tau\in [0,t]}\left(\Vert A(\tau,\bullet)\Vert_{C^\alpha(\Omega)}+\Vert \beta(\tau,\bullet)\Vert_{C^\alpha(\Omega)}\right)\vert x-y\vert^\alpha.
    \end{split}
\end{equation*}
Thus $\Phi$ is $C^{k,\alpha}$ with respect to $x$.

Since 
\begin{equation*}
    \partial_t\Phi=F\circ \Phi
\end{equation*}
and $F$ is of class $C^k$, we see that $\Phi$ is of class $C^{k+1}$ in time. Using the ODE
\begin{equation*}
    \partial_t\Phi=F\circ\Phi
\end{equation*}
we have that
\begin{equation*}
    \begin{split}
        \partial_t^{k+1}\Phi=R(F,\mathrm{d}F,\dots,\mathrm{d}^kF)\circ\Phi,
    \end{split}
\end{equation*}
where $R$ is a polynomial. Suppose without loss of generality that $s<t$. Then
\begin{equation*}
    \begin{split}
        \vert \partial_t^k\Phi(t,x)-\partial_t^k\Phi(s,x)\vert&\leq \int_s^t\vert (R(F,\mathrm{d}F,\dots,\mathrm{d}^kF)\circ\Phi)(\tau,x)\vert\ \mathrm{d}\tau\lesssim \Vert F\Vert^N_{C^k(\Bar{\Omega})}\vert t-s\vert
    \end{split}
\end{equation*}
for some positive power $N\in\mathbb{N}$,
telling us that $\partial_t^k\Phi$ is Lipschitz in time. This lets us conclude that $\Phi\in C^{k,\alpha}(D)$.
\end{proof}

With this we can obtain

\begin{lemma} \label{lemma_geodesic_ckminusone}
Let $g \in C^k$ (resp.\ $g \in C^{k,\alpha}$) where $k \geq 2$ and $0 < \alpha < 1$. The set 
\[
S = \{ (x,v,t) \,:\, (x,v) \in T \Omega, \ t \in I_{x,v} \},
\]
is open, and the maps 
\[
(x,v,t) \mapsto \gamma_{x,v}(t), \qquad (x,v,t) \mapsto \dot{\gamma}_{x,v}(t)
\]
are $C^{k-1}$ (resp.\ $C^{k-1,\alpha}$) from $S$ to $\mathbb{R}^n$.
\end{lemma}
\begin{proof}
If $g \in C^k$ for $k \geq 2$, then it follows from Lemma \ref{lemma_flow_regularity} that given any compact $K \subset T \Omega$, the map $(x,v,t) \mapsto z_{x,v}(t)$ is $C^{k-1}$ jointly in $(x,v) \in U_{K}$ and $t \in (-\varepsilon_{K}, \varepsilon_{K})$.  Note that this means that both $\gamma_{x,v}(t)$ and $\dot{\gamma}_{x,v}(t)$ are $C^{k-1}$ in $(x,v,t)$.

It follows from the above that if $[a,b] \subset I_{x_0,v_0}$, then $[a,b] \subset I_{x,v}$ for $(x,v)$ near $(x_0,v_0)$ (see e.g.\ the argument in \cite[Lemma 9 in Section 5.2]{petersen2006riemannian}). Thus $S$ is an open set. Applying Lemma \ref{lemma_flow_regularity} again shows that $\gamma_{x,v}(t)$  and $\dot{\gamma}_{x,v}(t)$ are $C^{k-1}$ for $(x,v,t)$ in $S$.

Similarly, if $g \in C^{k,\alpha}$, the claim follows from Lemma \ref{lemma_flow_regularity}.
\end{proof}

The geodesic equation may be written in terms of the Levi-Civita connection as 
\[
D_t \dot{\gamma}(t) = 0,
\]
and taking the time derivative of $\langle \dot{\gamma}(t), \dot{\gamma}(t) \rangle$ shows that 
\[
|\dot{\gamma}_{x,v}(t)| = |v|.
\]
For any $x$, the domain of the exponential map is 
\[
D_x = \{ v \in \Omega \,:\, [0,1] \subset I_{x,v} \}.
\]
Uniqueness gives the scaling property
\[
\gamma_{x,\lambda v}(t) = \gamma_{x,v}(\lambda t), \qquad \lambda > 0.
\]
This implies that $D_x$ is a star-shaped open set, and the exponential map is defined by 
\[
\exp_x(v) = \gamma_{x,v}(1), \qquad v \in D_x.
\]
By Lemma \ref{lemma_geodesic_ckminusone}, the set $D = \cup_{x \in \Omega} D_x$ is an open subset of $T \Omega$ and the exponential map has the following regularity properties.

\begin{lemma}
If $g$ is $C^k$ (respectively $C^{k,\alpha}$) with $k \geq 2$ and $0 < \alpha < 1$, then the map 
\[
E: D \to \mathbb{R}^n, \ \ E(x,v) = \exp_x(v)
\]
is $C^{k-1}$ (respectively $C^{k-1,\alpha}$).
\end{lemma}

Next we state the Gauss lemma.

\begin{lemma}
Let $g$ be a $C^2$ Riemannian metric. Then 
\[
\langle d \exp_x|_{tv}(v), d \exp_x|_{tv}(w) \rangle = \langle v, w \rangle, \qquad tv \in D_x, \ w \in T_x \mathbb{R}^n.
\]
\end{lemma}
\begin{proof}
Define $\Gamma(t,s) = \exp_x(tv + s tw)$ for $s$ small, so $\Gamma$ is $C^1$ and also $\partial_t \Gamma(t,s) = \dot{\gamma}_{x,v+sw}(t)$ is $C^1$ in $t$ and $s$. The quantity $\partial_s \Gamma(t,s)$ is only $C^0$. However, the symmetry property $\Gamma_{jk}^l = \Gamma_{kj}^l$ and a standard computation \cite[Lemma 6.3]{lee2018introduction} show that in the sense of distributions 
\[
D_t \partial_s \Gamma(t,s) = D_s \partial_t \Gamma(t,s)
\]
where the right hand side is in $C^0$. This proves that also $\partial_s \Gamma(t,0)$ is $C^1$ in $t$.

We have 
\[
t \langle d \exp_x|_{tv}(v), d \exp_x|_{tv}(w) \rangle = \langle \partial_t \Gamma(t,0), \partial_s \Gamma(t,0) \rangle.
\]
Using that $\Gamma(t,0)$ is a geodesic and invoking the symmetry property above, the $t$-derivative of the right hand side is 
\begin{align*}
\partial_t \langle \partial_t \Gamma(t,0), \partial_s \Gamma(t,0) \rangle &= \langle \partial_t \Gamma(t,0), D_t \partial_s \Gamma(t,0) \rangle \\
 &= \frac{1}{2} \partial_s \langle \partial_t \Gamma(t,s), \partial_t \Gamma(t,s) \rangle|_{s=0}.
\end{align*}
Now each curve $t \mapsto \Gamma(t,s)$ is a geodesic, hence has constant speed $|v+sw|$, so the expression above is equal to $\langle v, w \rangle$. Thus $\langle \partial_t \Gamma(t,0), \partial_s \Gamma(t,0) \rangle = t\langle v, w \rangle + a$ for some constant $a$, and evaluating  at $t=0$ gives $a=0$.
\end{proof}

The Gauss lemma allows us to prove that among exponential images of curves in $D_x$ with the same endpoints, radial geodesics are the shortest.

\begin{lemma} \label{prop_expx_domain_minimizing}
Let $g$ be a $C^2$ Riemannian metric, let $x \in \Omega$ and $w \in D_x$, let $\eta_0: [0,1] \to D_x$ be the curve $\eta_0(t) = tw$, and let $\eta: [0,1] \to D_x$ be any $C^1$ curve with $\eta(0) = 0$ and $\eta(1) = w$. Then 
\[
\int_0^1 |(\exp_x \circ \,\eta_0)'(t)| \,\mathrm{d}t \leq \int_0^1 |(\exp_x \circ \,\eta)'(t)| \,\mathrm{d}t
\]
with equality if and only if $\eta$ is a reparametrization of $\eta_0$.
\end{lemma}
\begin{proof}
We may assume that $w \neq 0$ and $\eta(t) \neq 0$ for $0 < t \leq 1$ (if not, let $t_0$ be the last time with $\eta(t_0) = 0$ and replace $\eta$ by $\eta|_{[t_0,1]}$ rescaled to the interval $[0,1]$). We write $\eta(t) = r(t) \omega(t)$ where $r(t) = |\eta(t)|$ and $|\omega(t)| = 1$. Then for $t > 0$ one has 
\[
\dot{\eta}(t) = \dot{r}(t) \omega(t) + r(t) \dot{\omega}(t).
\]
The condition $|\omega(t)| = 1$ implies $\langle \omega(t), \dot{\omega}(t) \rangle = 0$. Using the Gauss lemma, we obtain that 
\begin{gather*}
\langle d \exp_x|_{\eta(t)} (\omega(t)), d \exp_x|_{\eta(t)} (\dot{\omega}(t)) \rangle = 0, \\
|d \exp_x|_{\eta(t)} (\omega(t))| = |d \exp_x|_{\eta(t)} (\dot{\omega}(t))| = 1.
\end{gather*}
Combining these facts gives that 
\[
|(\exp_x \circ \,\eta)'(t)|^2 = |d \exp_x|_{\eta(t)}(\dot{\eta}(t))|^2 \geq \dot{r}(t)^2.
\]
Thus the lengths satisfy 
\begin{align*}
\int_0^1 |(\exp_x \circ \,\eta)'(t)| \,\mathrm{d}t &\geq \int_0^1 |\dot{r}(t)| \geq r(1) - r(0) = |w| \\
 &= \int_0^1 |(\exp_x \circ \,\eta_0)'(t)| \,\mathrm{d}t.
\end{align*}
Equality holds if and only if $\dot{\omega}(t) = 0$ and $\dot{r}(t) \geq 0$, which corresponds to the case where $\eta$ is a reparametrization of $\eta_0$.
\end{proof}

As a consequence, we obtain that short geodesics are unique length minimizing curves.

\begin{lemma} \label{lemma_expx_ball}
Let $g$ be a $C^2$ Riemannian metric, let $x \in \Omega$, and suppose that 
\[
\exp_x: B(0,r) \to U
\]
is a bijective map, where $B(0,r) \subset T_x \Omega$ is the ball of radius $r$ and $U \subset \Omega$ is open. Then 
\[
U = \{ y \in \Omega \,:\, d(x,y) < r \}.
\]
Moreover, for any $v \in B(0,r)$ the curve $\gamma_{x,v}|_{[0,1]}$ is the unique (up to reparametrization) length minimizing curve from $x$ to $\exp_x(v)$ among all piecewise $C^1$ curves in $\Omega$.
\end{lemma}
\begin{proof}
Since $\exp_x$ is $C^1$, the inverse function theorem shows that $\exp_x: B(0,r) \to U$ is a $C^1$ diffeomorphism. Fix $v \in B(0,r)$, let $y = \exp_x(v)$, let $\gamma = \gamma_{x,v}|_{[0,1]}$ be the radial geodesic from $x$ to $y$, and let $\eta: [0,1] \to \Omega$ be any piecewise $C^1$ curve from $x$ to $y$.

If $\eta$ is contained in $U$, then by the diffeomorphism property $\eta$ is the exponential image of a curve in the domain of $\exp_x$. By Lemma \ref{prop_expx_domain_minimizing}, the length of $\gamma$ (which is equal to $|v|$) is less than or equal to the length of $\eta$, with equality if and only if $\eta$ is a reparametrization of $\gamma$. On the other hand, if $\eta$ leaves $U$, let $a \in [0,1]$ be the largest time with $\eta(a) = x$, and let $t_0 \in (a,1]$ be the smallest time with $\eta(t_0) \notin U$. By Lemma \ref{prop_expx_domain_minimizing} $\eta|_{[a,t_0]}$ has length $\geq r > |v|$, so $\eta$ is not shorter than $\gamma$. This proves that $\gamma$ is the unique (up to reparametrization) length minimizing curve from $x$ to $y$.

It is evident that $\exp_x(B(0,r)) \subset \{ y \in \Omega \,:\, d(x,y) < r \}$. For the converse, if $d(x,y) < r$, then there is a curve $\eta$ from $x$ to $y$ with length $< r$. The argument above shows that $\eta$ must be a radial geodesic, so the converse inclusion also holds.
\end{proof}

Finally, we give a representation of the Riemannian distance function that will be needed for our arguments.

\begin{lemma} \label{lemma_distance}
    Let $g \in C^{k,\alpha}(\Omega)$ where $k \geq 2$ and $0 < \alpha < 1$. In some neighborhood of the diagonal $\{ (x,x) \,:\, x \in \Omega \}$, we have 
\begin{equation*}
    d(x,y)^2=G_{pq}(x,y)(x-y)^p(x-y)^q
\end{equation*}
where $G\in C^{k-2,\alpha}$ and $G(x,x)=g(x)$.

Moreover, if $(\overline{\Omega},g)$ is simple with $g \in C^{k,\alpha}(\overline{\Omega})$, then $d(x,y)^2$ is in $C^{k-1,\alpha}(\overline{\Omega} \times \overline{\Omega})$ and the above representation holds with $G \in C^{k-2,\alpha}(\overline{\Omega} \times \overline{\Omega})$.
\end{lemma}
\begin{proof}
We consider the map $F(x,v) = (x, E(x,v))$. Then $F$ is $C^{k-1,\alpha}$, $F(x,0) = (x,x)$, and 
\[
DF(x,0) = \begin{pmatrix} I & 0 \\ D_x E(x,0) & I \end{pmatrix}
\]
using that $D_v E(x,0) = D \exp_x|_0 = I$. By the inverse function theorem, $F$ is a $C^{k-1}$ diffeomorphism from some neighborhood of $(x,0)$ onto a neighborhood of $(x,x)$. Since $D(F^{-1})(z) = DF(F^{-1}(z))^{-1}$, we see that $F^{-1}$ is $C^{k-1,\alpha}$. It follows that there is a neighborhood $W$ of the diagonal in $\Omega \times \Omega$ and a $C^{k-1,\alpha}$ map $v: W \to \mathbb{R}^n$ such that $F(x, v(x,y)) = (x,y)$. By Lemma \ref{lemma_expx_ball}, after shrinking $W$ if necessary, we can arrange that $\gamma_{x,v(x,y)}$ is the minimizing geodesic from $x$ to $y$ when $(x,y) \in W$.

Next we observe that for $(x,y) \in W$, one has 
\[
d(x,y)^2 = |v(x,y)|_{g(x)}^2 = g_{jk}(x) v^j(x,y) v^k(x,y).
\]
Since $v$ is $C^{k-1,\alpha}$ and $v(x,x) = 0$, we have the Taylor expansion 
\[
v^j(x,y) = a^j_p(x,y)(x-y)^p,
\]
where $a^j_p \in C^{k-2,\alpha}(W)$. Moreover, since $E(x, v(x,y)) = y$, we have $$I = D_y(E(x, v(x,y)))|_{y=x} = (D \exp_x)_0(D_y v|_{y=x}) = D_y v|_{y=x}$$ and hence we have $a^j_p(x,x) = \partial_p v^j(x,x) = \delta_p^j$. It follows that 
\[
d(x,y)^2 = G_{pq}(x,y)(x-y)^p (x-y)^q,
\]
where $G_{pq}(x,y) = g_{jk}(x) a^j_p(x,y) a^k_q(x,y)$ is $C^{k-2,\alpha}$ in $W$ and satisfies $G_{pq}(x,x) = g_{pq}(x)$.

Suppose now that $(M,g)$ is simple, where $M = \overline{\Omega}$. By the definition of a simple manifold, there is $U_1 \supset M$ such that $(\overline{U}_1, g)$ is simple. Fix $x_0, y_0 \in U_1$ and let $y_0 = E(x_0,v_0)$ where $v_0 \in V_{x_0}$. If $F$ is the map above, we have 
\[
DF(x_0,y_0) = \begin{pmatrix} I & 0 \\ D_x E(x_0,v_0) & D_v E(x_0,v_0) \end{pmatrix}
\]
Here $D_v E(x_0,v_0) = (D \exp_{x_0})(v_0)$ is invertible since $(\overline{U}_1, g)$ is simple. Hence the inverse function theorem gives a $C^{k-1,\alpha}$ function $v(x,y)$ near $(x_0,y_0)$ such that 
\[
E(x,v) = y \text{ for $(x,v)$ near $(x_0,v_0)$ and $y$ near $y_0$} \quad \Longleftrightarrow \quad v = v(x,y).
\]
By the assumption that $\exp_x$ is a diffeomorphism onto $\overline{U}_1$, Lemma \ref{prop_expx_domain_minimizing} ensures that the curve $t \mapsto E(x,t v(x,y))$ is the only geodesic in $\overline{U}_1$ joining $x, y \in U_1$ and therefore 
\[
d(x,y)^2 = |v(x,y)|_{g(x)}^2.
\]
It follows that $d(x,y)^2 \in C^{k-1,\alpha}(U_1 \times U_1)$. From the Taylor expansion argument above, we obtain that $G_{pq} \in C^{k-2,\alpha}(U_1 \times U_1)$.
\end{proof}

\bibliographystyle{alpha}
\bibliography{ref}

\end{document}